\documentclass[a4paper,11pt,reqno]{amsart}

\pdfoutput=1

\usepackage[utf8]{inputenc}
\usepackage[T1]{fontenc}
\usepackage{lmodern}
\usepackage[french,english]{babel}
\usepackage{amsthm, amssymb, amsmath, amsfonts, mathrsfs, dsfont, esint,textcomp}
\usepackage[colorlinks=true, pdfstartview=FitV, linkcolor=blue, citecolor=blue, urlcolor=blue,pagebackref=false]{hyperref}
\usepackage{mathtools}

\usepackage{a4wide}  

\usepackage{color}

\usepackage{microtype}

\newtheorem{thm}{Theorem}[section]

\newtheorem{lem}[thm]{Lemma}

\theoremstyle{remark}

\theoremstyle{definition}

\renewcommand{\leq}{\leqslant}
\renewcommand{\geq}{\geqslant}

\renewcommand{\subset}{\subseteq}

\newcommand{\uh}{u_{\mathsf{h}}}

\newcommand{\N}{\mathbb{N}}

\newcommand{\1}{\mathbf{1}}
\newcommand{\R}{\mathbb{R}}

\newcommand{\Z}{\mathbb{Z}}
\newcommand{\X}{\mathcal{X}}

\renewcommand{\P}{\mathbb{P}}

\newcommand{\eps}{\varepsilon}

\newcommand{\oh}{\frac 1 2}

\newcommand{\aeps}{\eps^\frac{d}{d-2}}
\newcommand{\dd}{\, d}
\newcommand{\wto}{\rightharpoonup}
\newcommand{\Rd}{\R^d}
\newcommand{\Zd}{\Z^d}

\DeclareMathOperator{\dist}{dist}
\DeclareMathOperator{\capacity}{Cap}

\mathtoolsset{showonlyrefs}

\title[]{Homogenization for the Poisson equation in randomly perforated domains under minimal assumptions on the size of the holes}
\author{Arianna Giunti, Richard H\"ofer, Juan J. L. Vel\'azquez}

\address[Arianna Giunti]{Institute for Applied Mathematics, University of Bonn, Endenicher Allee 60, 53115 Bonn, Germany}
\address[Richard H\"ofer]{Institute for Applied Mathematics, University of Bonn, Endenicher Allee 60, 53115 Bonn, Germany}
\address[Juan J.L. Vel\'azquez]{Institute for Applied Mathematics, University of Bonn, Endenicher Allee 60, 53115 Bonn, Germany}

\begin{document}

\begin{abstract}
This paper deals with the homogenization of the Poisson equation in a bounded domain of $\Rd$, $d>2$, which is perforated by a random number of small spherical holes 
with random radii and positions. We show that for a class of stationary short-range correlated measures for the centres and radii of the holes, we recover in the homogenized 
limit an averaged analogue of the ``strange term'' obtained by Cioranescu and Murat in the periodic case [D. Cioranescu and F. Murat, \textit{Un term \'etrange venu d'ailleurs} (1986)]. We stress that we only require that the random radii have finite $(d-2)$-moment, which is the minimal assumption in order to ensure that the average of the capacity of the balls is finite. Under this assumption, there are holes which overlap with probability one. However, we show that homogenization occurs and that the clustering holes do not have any effect in the resulting homogenized equation.

\end{abstract}
\maketitle
%
%
%
%
%
%
%
%

\section{Introduction}

In this paper, we study the homogenization of the Poisson equation 
\begin{align}\label{laplace.eps}
\begin{cases}
-\Delta u_\eps = f \ \ \ \ &\text{ in $D^\eps$}\\
\, u_\eps = 0 \ \ \ \ &\text{ on $\partial D^\eps$},
\end{cases}
\end{align}
where the domain $D^\eps$ is obtained by removing from a bounded set $D \subset \Rd$, $d>2$, the union of properly rescaled spherical holes: Given a collection of points $\Phi= \{ z_i \}_{i\in \N} \subset \Rd$ and associated radii $\{\rho_i\}_{z_i \in \Phi} \subset \R_+$, we define
\begin{align}
\label{abstract.holes}
D^\eps:= D \backslash \bigcup_{z_i \in \Phi \cap \frac 1 \eps D} B_{\aeps \rho_i}( \eps z_i),
\end{align}
where $\frac 1 \eps D:= \{ 	x \in \Rd \, \colon \, \eps x \in D \}$.

\smallskip

In this paper, we assume that $D^\eps$ is a random set. More precisely, we assume that the collection $\Phi$ of the centres is generated according to a stationary point process on $\Rd$ and that the radii $\{\rho_i\}_{z_i \in \Phi}$ are unbounded random variables with short-range correlations. We show that, $\P$-almost surely, when $\eps \downarrow 0^+$, the solutions of \eqref{laplace.eps} weakly converge in $H^1_0(D)$ to the solution of
\begin{align}
\label{laplace.hom}
\begin{cases}
(-\Delta  + C_0) \uh = f \ \ \ \ &\text{ in $D$}\\
\, \uh = 0 \ \ \ \ &\text{ on $\partial D$}.
\end{cases}
\end{align}
Here, the constant $C_0> 0$ may be expressed in terms of an averaged density of capacity generated by the holes. We thus recover in the limit the analogue of 
the well-known ``strange term'' obtained by Cioranescu and Murat in the case of deterministic and periodic holes \cite{CioranescuMurat}. In this latter case,
which is equivalent to taking $\Phi = \Zd$ and $\rho_i \equiv r$ for $r >0$, $C_0$ equals the capacity of a ball of radius $r$. More precisely,
\begin{align}
	\label{capacity.deterministic}
	C_0= (d-2) \mathcal{H}^{d-1}( S^{d-1}) r^{d-2},
\end{align}
where $S^{d-1}$ is the $(d-1)$-dimensional unit sphere in $\Rd$.

\bigskip

The main contribution of this paper is to show that the homogenization of $u_{\varepsilon}$ to $\uh$ of \eqref{laplace.hom} takes place under minimal assumptions
on the decay
of the density functions for the radii $\{\rho_{i}\}_{z_i \in \Phi}$ of the holes. Namely, we only assume that
the average capacity of each hole is finite. More precisely, we will
show that the homogenization result holds assuming that the
configurations of the holes in $\mathbb{R}^{d}$, $d>2$, are assigned to a class of
probability measures for which the expectation of each radius $\rho_i$ satisfies
only%
\begin{equation}
\langle \rho^{d-2}_i\rangle <\infty.\label{AverRadii}
\end{equation}

 In view of \eqref{capacity.deterministic}, this is the minimal assumption under which one can expect $C_{0}$ in \eqref{laplace.hom} to be finite 
and therefore the limit problem \eqref{laplace.hom} to be meaningful.  On the other hand, to assume only
(\ref{AverRadii}) on the hole distributions poses some difficulties due to
the fact that the presence of balls in \eqref{abstract.holes} with large radii could allow the onset of clusters having large capacity. These
clusters might prevent the convergence of $u_{\varepsilon}$ to the solutions
of the homogenized problem $\uh$. In fact, it is known that the onset of
large clusters take place with probability one in systems of spherical holes
filling a small volume fraction of the space $\Rd$, if the radii
are distributed according to probability distributions with sufficiently fat tails (cf.
\cite{grimmett1999percolation, ContinuumPercolation}). However, in \cite{grimmett1999percolation, ContinuumPercolation} the holes are not rescaled 
as in \eqref{abstract.holes}. In this paper, we prove that with the rescaling of the balls as in
\eqref{abstract.holes}, the assumption (\ref{AverRadii}) is sufficient to ensure that  no percolating-like
structures appear in the limit $\varepsilon\downarrow 0^+$.

\bigskip

Stochastic homogenization problems for \eqref{laplace.eps} have been considered before in
the literature. The earliest results were the ones in \cite{MarchenkoKhruslov}
and \cite{papvar.tinyholes}. In \cite{papvar.tinyholes}, the evolution version of \eqref{laplace.eps}
(i.e. the linear heat equation) is considered, and it is shown that the corresponding
solutions $u_\eps$ converge to the analogous evolution version of $\uh$ assuming that
$\varepsilon^{-d}$ spherical holes are distributed independently according to some density function $V \in C_{0}^{\infty}(
\mathbb{R}^{d}).$ All the holes are assumed to have constant and fixed radius $r_{j}=\varepsilon^{\frac{d}{d-2}%
}.$ 

In \cite{MarchenkoKhruslov}, the authors consider spherical hole configurations constituted of
$\varepsilon^{-3}$  balls in a bounded domain $D\subset\mathbb{R}^{3}$,
selected according to some classes of probability measures in which the
balls cannot overlap due to the presence of a hard-sphere potential. These
probability measures allow also to have short range correlations between two holes. As in our paper, also in \cite{MarchenkoKhruslov}, the balls have random size $\aeps\rho_i$,
where the random variables $\rho_i$ are assumed to satisfy
\begin{equation}\label{RadiiMK}
\langle \rho_i^{3+\beta}\rangle <\infty\text{ with }\beta
>0.
\end{equation}
Under these assumptions, which we further discuss below, it is proved in \cite{MarchenkoKhruslov} that the
solutions of \eqref{laplace.eps} converge to the solutions of \eqref{laplace.hom}.

Stochastic homogenization for equations related to \eqref{laplace.eps} and domains $D^\eps$ perforated as in \eqref{abstract.holes} has been considered also
more recently in \cite{CaffarelliMellet, CasadoDiaz.nonlinear, CasadoDiaz}. In
\cite{CaffarelliMellet}, the authors consider the
homogenization limit for the obstacle problem associated to a Dirichlet
functional in $D \subset \mathbb{R}^{d},$ $d\geq2$, in which the solutions must satisfy $u_{\varepsilon} \geq 0$ in the collection of small compact sets $D \backslash D^\eps$.
Differently from our setting, the compact sets constituting $D \backslash D^\eps$ are centred on a periodic lattice, but they can have random shapes which are uniformly bounded by $M \aeps$, for a fixed constant $M > 0$. 
Assuming an ergodicity condition of the probability measure on the shapes of the sets, the authors of \cite{CaffarelliMellet} prove that the minimizers $u_\eps$ converge to the solution $\uh$ of the
semi-linear equation $-\Delta \uh +\alpha\left( \uh \right)_{-}=f$ in $D$.

In \cite{CasadoDiaz.nonlinear, CasadoDiaz}, the main focus is to study the stochastic homogenization of elliptic equations which not only include
\eqref{laplace.eps} but also singular elliptic operators like the $p-$Laplacian
operator. The probability measures considered in these works
allow to have hole configurations having random shapes which are uniformly bounded
by $\aeps$. Moreover, a stringent condition is assumed on the probability measure for the positions of the holes to ensure that the minimal distance between the holes is of order
$\varepsilon$ with probability one.

\bigskip

In all the papers listed above with the exceptions of \cite{MarchenkoKhruslov}, it is assumed
that the size of each hole is of order $\varepsilon^{\frac{d}{d-2}}$ with
probability one. We emphasize that the main technical difficulty in our paper
is due to the fact that under the sole assumption \eqref{AverRadii}, namely for distributions of the size of holes having fat tails
decreasing slowly enough, there exist, with probability one, domains $D^\eps$ with the form \eqref{abstract.holes} punctured by clusters of two or more overlapping holes. These
clusters do not occur (with probability tending to one as $\varepsilon\downarrow 0^+$) under the assumption \eqref{RadiiMK} which is made in \cite{MarchenkoKhruslov}.

\bigskip

In order to prove the homogenization results mentioned above there are
different methods in the literature. The first one, which was introduced by
Cioranescu and Murat in \cite{CioranescuMurat}, is related to the energy method of Tartar
\cite{TartarBook2009}. It is based on the construction of some \textit{oscillating test
functions} $w_{\varepsilon}$. A related approach has been used in the analysis
of several deterministic and stochastic homogenization problems (cf.
\cite{CaffarelliMellet,CasadoDiaz.nonlinear, CasadoDiaz, DalMasoGarroni.punctured}) and this
is also the approach that will be used in this paper. 

A second approach is based on the construction of suitable projection operators in
Hilbert spaces which are defined using the geometry of the perforated domains.
This approach was introduced by Marchenko and Khruslov (cf. \cite{MarchenkoKhruslov} and the references therein). Related to this approach is
the so-called method of reflections which yields
 a formal series representation for the Dirichlet problem \eqref{laplace.eps} that is frequently used in the physics literature. In
 \cite{HoeferVelazquez.reflections}, the authors gave a
precise meaning to some of these formal series and used this method to prove homogenization results.

A third approach, used for instance in \cite{papvar.tinyholes}, employs the probabilistic interpretation of the solutions $u_\eps$ of 
\eqref{laplace.eps}  (and its evolution analogue), in terms of the properties of the
Brownian motion. In particular, the solutions of \eqref{laplace.eps} as well as the term $C_0$ arising in the limit equation can be obtained in terms of
expectations of functions of the survival time of a Brownian walker among
obstacles. 

Finally, we also mention that for problems related to \eqref{laplace.eps}, a different approach has been introduced in \cite{Niethammer.homogenization, NiethammerVelazquez.coarsening1, NiethammerVelazquez.coarsening2}. In this series of papers, the main goal is to study a dynamical version of \eqref{laplace.eps}, where the holes evolve according to the function $u_\eps$ itself. In this case, the main challenge is thus to obtain estimates for the solution in the space $L^{\infty}$ instead of the Sobolev space $H^1$. The starting point used in \cite{NiethammerVelazquez.coarsening1, NiethammerVelazquez.coarsening2} is an ansatz for the structure of the solution of \eqref{laplace.eps} which gives rise to an explicit expression for an approximate solution of  \eqref{laplace.eps}. 
The difference between this approximate solution and the solution of \eqref{laplace.eps}  is
then estimated using the maximum principle. 

Stochastic homogenization results have been obtained using this
approach in \cite{NiethammerVelazquez.coarsening2}, and in the case of solutions of \eqref{laplace.eps} in unbounded domains, they rely on the study of screening
properties \cite{NiethammerVelazquez.screening}. Concerning the introduction and the study of such screening phenomena for interacting particles, we also refer to 
\cite{Niethammer.homogenization} and \cite{NiethammerOtto.screening}.

\bigskip

In the problems of stochastic homogenization, two different types of
convergence results are obtained. One approach consists in introducing a
probability measure $\mathbb{P}$ on the space of hole configurations
$\Omega$ (positions and shapes) in $\mathbb{R}^{d}.$ The Dirichlet
problem \eqref{laplace.eps} is then solved for each fixed realization $\omega\in\Omega$ in
a bounded domain which is obtained by means of \eqref{abstract.holes}. It is then proved
that $u_{\varepsilon}$ converges for $\mathbb{P-}a.s.$ as
$\varepsilon\downarrow 0^+$ to the solutions of \eqref{laplace.hom}. This is the type of
results obtained in \cite{CaffarelliMellet,  CasadoDiaz.nonlinear, CasadoDiaz}, and also in
this paper. 

The second approach consists in creating  configurations
containing $\varepsilon^{-d}$ holes in a bounded domain according to a
family of probability measures $\mathbb{P}_{\varepsilon}$ defined on a space
of configurations $\Omega_{\varepsilon}.$  The homogenization results is thus expressed in terms of convergence in probability, namely that for any
$\delta>0$, $\lim_{\varepsilon\downarrow 0^+}\mathbb{P}_{\varepsilon}\left( \left\{  \left\Vert u_{\varepsilon}-u\right\Vert >\delta\right\}  \right)=0$, where $\left\Vert \cdot\right\Vert $ is a suitable norm. The results obtained in \cite{MarchenkoKhruslov}, \cite{NiethammerVelazquez.coarsening2} and  \cite{papvar.tinyholes} are of this type.

\subsection{Main ideas and organisation of the paper}\label{main.ideas}
As already mentioned in the previous discussion, in this paper, we focus on probability measures where the radii of the balls in \eqref{abstract.holes} satisfy merely the minimal condition \eqref{AverRadii} on their moments,
and the centres of the balls are distributed according to a stationary point process on the whole space. We allow that both the centres and the radii have short-range correlations. This class of measures
includes the cases of balls having independent and identically distributed radii and centres either periodic or distributed according to an homogeneous Poisson point
process (cf. settings (a) and (b) in the next section). We also give some explicit examples of short-range correlated measures which are constructed starting from clustering 
or repulsive point processes for the centres of the holes (cf. setting (c) in the next section). 

In order to prove the main homogenization result for these measures, we adapt the argument of \cite{CioranescuMurat} to translate the conditions on the geometry of
the holes of $D^\eps$ into properties of the associated \textit{oscillating test function} $w_\eps$.
These functions account for the presence of the holes in the domain $D^\eps$ by correcting any admissible test function $\phi \in C^\infty_0(D)$ for \eqref{laplace.hom}
into  an admissible test function $w_\eps \phi \in H^1_0( D^\eps)$ for \eqref{laplace.eps}. The main breakthrough of \cite{CioranescuMurat} 
is the formulation of sufficient conditions on $w_\eps$ which allow to treat the error terms generated by the presence of $w_\eps$ in the weak formulation for 
\eqref{laplace.eps}. In the limit $\eps \downarrow 0^+$, these errors are the ones giving rise to the additional term $C_0 \uh$ in \eqref{laplace.hom}.
In the case of periodic balls in $\R^d$, the authors in \cite{CioranescuMurat} explicitly construct the test functions $w_\eps$ and obtain \eqref{laplace.hom} with the value for $C_0$ given by \eqref{capacity.deterministic}; from this construction, the link between the term $C_0$ and the density of capacity generated by the holes becomes apparent and motivates the necessary
choice in \eqref{abstract.holes} of the length-scales $\aeps$ for the radii and $\eps$ for the distance between the centres.

The main challenge of this paper is that with the sole assumption \eqref{AverRadii} on the radii, we need to deal in (almost) all the configurations with the presence
of large radii. In spite of the scaling of \eqref{abstract.holes}, the associated big balls may overlap and potentially break down the construction of the functions
$w_{\varepsilon}$. The main idea of our proof is to show that, even though with probability one there are regions where the balls overlap,
the moment assumption on the radii is sufficient to ensure that almost surely these regions have a capacity which vanishes in the limit $\eps \downarrow 0^+$.
This yields that the contribution of the functions $w_\eps$ to the new term in the limit equation is restricted to the region of the domain $D^\eps$ where the balls 
are small and well-separated.

\medskip

The structure of this paper is the following: In the next section, we give a precise definition of the processes generating the holes in 
\eqref{abstract.holes} and introduce some examples which are included in our setting; we then state the main homogenization result
(Theorem \ref{t.main}). 
Section \ref{proof.main} contains the proof of the theorems provided that the oscillating test functions exist, while Section \ref{section.oscillating} is devoted to 
the crucial arguments for the construction of such oscillating test functions. Section \ref{aux} provides some probabilistic results for \textit{marked point processes} on
which the previous section relies and which make the arguments of this paper totally self-contained. To this purpose, we also include an appendix with a proof of a Strong
Law of Large Numbers which is tailored to the processes that we consider. In the sake of what we think is a more comfortable reading,
we do not prove our main result directly for a general probability measure, but we first give the argument in the case of holes with periodic centres and independent and identically distributed radii. By relying on the abstract results of Section \ref{aux}, we then show how to adapt this proof to a general measure with short-range correlations. This, we believe, gives a more intuitive structure to the arguments of this paper.

\section{Setting and main result}\label{setting}

Let $D \subset \R^d$, $d >2$, be an open and bounded set that it is star-shaped with respect to the origin\footnote{This assumption ensures that the sets in the family $\{\frac{1}{ \eps} D \}_{\eps > 0}$ (see definition after formula \eqref{holes}) are nested. This is not a necessary condition for our results to hold, but it avoids some technicalities in our proof and it keeps our arguments and our notation leaner.}. For $\varepsilon > 0$, we denote by $D^\varepsilon \subset D$ the domain obtained by 
removing from $D$ the closure of a set of ``small'' holes $H^\varepsilon \subset \R^d$ of the form:
\begin{equation}\label{holes}
H^\varepsilon := \bigcup_{z_j \in \Phi \cap \frac 1 \eps D} B_{\varepsilon^{\frac{d}{d-2}} \rho_j} (\varepsilon z_j),
\end{equation}
where $\frac 1 \eps D := \{ x \in \Rd \colon \, \eps x \in D \}$, the set $\Phi \subset \R^d$ is a random collection of (countably many) points and the radii
$\{\rho_i\}_{z_i \in \Phi} \subset \R^+$ are random variables. The set $H^\eps$ may thus be thought as being generated  by a marked point process
$(\Phi, \mathcal{R})$ on $\Rd \times \R_+$, where  $\Phi$ is a point process on $\Rd$ for the centres of the balls, and the marks
$\mathcal{R}=\{ \rho_i\}_{z_i \in \Phi} \subset \R_+$ are the radii associated to each centre. We refer to \cite[Chapter 9, Definitions 9.1.I - 9.1.IV]{Daley.Jones.book2} 
for a rigorous definition of marked point processes as a class of random measures on $\Rd \times \R_+$. We remark indeed that there is a one-to-one correspondence between
representing each realisation of the process as a collection of points and radii $\{ (z_i, \rho_i) \}_{i\in \N} \subset \Rd \times \R_+$ as we do in this paper, and as the
atomic measure $\mu:=  \sum_{i\in \N} \delta_{(z_i, r_i )}$ on $\Rd \times \R_+$. We also note that both the previous representations are invariant under permutation of the indices
$i \in \N$ and thus that there is no preferred ordering of the centres of the balls generating $H^\eps$.

\medskip

We denote by $(\Omega, \mathcal{F}, \P)$ the probability space associated to the process $(\Phi, \mathcal{R})$ and, for every $\omega \in \Omega$, we write $H^\eps( \omega)$ for the set defined in \eqref{holes} with 
$(\Phi, \mathcal{R})(\omega)$.  Throughout this paper, we assume that $(\Phi, \mathcal{R})$ satisfies the the following properties:

 \begin{itemize}
 
 \item The process $\Phi$ is stationary: For every $x\in \Rd$ we have $\tau_x  \circ \Phi \stackrel{\mathcal{L}}{=} \Phi$, where for each
 $\{ z_i\}_{i\in \N} \subset \Rd$ the translations are defined as
 \begin{align}\label{translations}
 \tau_{x}(\{z_i \}_{i\in \N})= \{z_i + x \}_{i\in \N}.
 \end{align}
 
 \smallskip
 
 \item There exists $\lambda< +\infty$ such that for any unitary cube $Q \subset \Rd$
 \begin{align}\label{finite.variance}
 \langle \#(\Phi \cap Q)^2 \rangle^{\frac 1 2} \leq \lambda ,
 \end{align}
where $\#(S) \in \N \cup \infty$ denotes the cardinality of a set $S$ and $\langle \cdot \rangle$ is the integration over $\Omega$ with respect
 to the measure $\P$. Note that, by stationarity of $\Phi$, the left-hand side of \eqref{finite.variance} does not depend on the position of $Q$.
 
 \smallskip
 
 \item The point process $\Phi$ satisfies a {\it strong mixing condition}: For any bounded Borel set $A \subset \R^d$, let $\mathcal{F}(A)$ be the smallest $\sigma$-algebra with respect to which the random variables $N(B)(\omega):= \#(\Phi \cap B)$ are measurable for every Borel set $B \subset A$. Then, there exist $C_1 < +\infty$ and $\gamma > d$ such that for every $ A\subset \Rd$ as above, every $x \in \R^d$ with $|x| > \operatorname{diam}(A)$ and every $\xi_1, \xi_2$ measurable with respect to $\mathcal{F}(A)$ and $\mathcal{F}(\tau_xA)$, respectively, we have
   \begin{align}\label{mixing}
 | \langle \xi_1 \xi_2 \rangle - \langle \xi_1\rangle \langle \xi_2 \rangle  | \leq  \frac{C_1}{1+(|x|- \operatorname{diam}(A))^\gamma} \langle \xi_1^2 \rangle^\oh \langle \xi_2^2 \rangle^\oh.
 \end{align}

 \smallskip
 
\item The marginal $\P_{\mathcal{R}}$ of the marks with respect to the process $\Phi$ has $1-$ and $2-$ correlation functions 
  \begin{align}\label{correlation.radii.0}
f_1((z, \rho)) &= h(\rho),\\
f_2(z_i, \rho_i , z_j, \rho_j) &= h( \rho_i) h(\rho_j) + g(|z_i - z_j|, \rho_i, \rho_j)  \ \ \ \ \forall i\neq j
  \end{align}
with 
\begin{align}\label{correlation.radii.2}
\int \rho^{d-2} h(\rho) d \rho < +\infty \ \ \ \ \ | g(r , \rho_1, \rho_2) | \leq  \frac{c}{(1+r^{\gamma})(1 + \rho_1^{p})(1 + \rho_2^{p})}
\end{align}
for $p > d-1,  \gamma > d$ and $c \in \R_+$.

The previous assumptions imply that the $(d-2)$-moment of the radii of the balls is finite and that, conditioned to the positions of the centres, the radii for two balls with centres in $z_1, z_2$ have correlations which vanish when the distance $|z_1-z_2| \to +\infty$. These correlations are short-range in the sense that the function $g$ above is integrable in the variable $r:=|z_1- z_2|$.
 \end{itemize}

\bigskip

Throughout this paper, we denote $D^\varepsilon(\omega):= D\backslash H^\eps(\omega)$ with $H^\eps(\omega)$ as in \eqref{holes}, and we identify any $v \in H^1_0(D^\eps(\omega))$ with the function $\tilde v \in H^1_0(D)$ obtained by extending $v$ as $v \equiv 0$ in $H^\eps(\omega)$. Then we have: 
\begin{thm}\label{t.main}
Let the holes in \eqref{holes} be generated by a marked point process $(\Phi, \mathcal{R})$. Let $\Phi$ satisfy \eqref{finite.variance} and \eqref{mixing}, and
let the marginal $\P_{\mathcal{R}}$ satisfy \eqref{correlation.radii.0} and \eqref{correlation.radii.2}. For $f \in H^{-1}(D)$ 
 and $\varepsilon > 0$,
let $u_\varepsilon= u_\varepsilon( \omega, \cdot ) \in H^1_0( D^\varepsilon(\omega) )$ solve \eqref{laplace.eps}. Then, there exist a constant $C_0> 0$ and $\uh \in H^1_0(D)$ solving 
\eqref{laplace.hom} such that for $\P$-almost every $\omega \in \Omega$
\begin{align}\label{convergence}
u_\varepsilon(\omega, \cdot) \rightharpoonup \uh \ \ \ \text{ in $H^1_0(D)$, \ \ for $\varepsilon \downarrow 0^+$.}
\end{align}

\smallskip

Moreover, we have that the constant $C_0$ in \eqref{laplace.hom} is defined as
\begin{align}\label{CM.term}
C_0= (d-2) \sigma_d \langle \, N(Q) \, \rangle \langle \, \rho^{d-2} \, \rangle,
\end{align}
where $\sigma_d = \mathcal{H}^{d-1}(S^{d-1})$ and $N(Q)$ is the number of centres falling into any fixed unitary cube $Q$.
\end{thm} 

\medskip

\subsection{Some examples of processes generating the holes $H^\eps$.}

Among the processes which satisfy the conditions required in the previous theorem, we mention the following three examples: For the first two examples, it is immediate that  \eqref{finite.variance}, \eqref{mixing}, \eqref{correlation.radii.0} and \eqref{correlation.radii.2} are satisfied. Also in the case of the examples given in (c), the previous conditions are satisfied and this follows by easy calculations that we postpone to the appendix.   

\begin{itemize}

\item[(a)] {\it The set $H^\eps$ is a collection of balls with periodic centres and i.i.d. radii.}\\ 
Here, the centres have deterministic positions $\Phi= \Zd$ and the marks $\mathcal{R}$ for the radii are a family of independent and identically distributed random variables which satisfy \eqref{AverRadii}. 
In this case, we have that the constant $C_0$ in \eqref{CM.term} reads $C_0= (d-2) \sigma_d \langle \, \rho^{d-2} \, \rangle$. 

\smallskip

\item[(b)] {\it The set $H^\eps$ is a collection of balls with centres generated by a Poisson point process and i.i.d. radii.} 
The process $\Phi$ is an homogeneous Poisson point process, and, conditioned to $\Phi$, the marks $\mathcal{R}$ are as in case (a). In this case, we have that $C_0= (d-2) \sigma_d \lambda \langle \, \rho^{d-2} \, \rangle$ with $\lambda> 0$ being the intensity of the Poisson point process $\Phi$.
\smallskip

\item[(c)] {\it The balls of $H^\eps$ have correlated radii and centres generated by a clustering or repulsive point process.} The process $\Phi$ is an attractive or a repulsive point processes with short-range correlations, respectively:

\smallskip

\begin{itemize}
\item[$(c.1)$]  Neymann-Scott cluster process on $\Rd$ (see, e.g. \cite[Example 6.3]{Daley.Jones.book}): Let $(\Phi_1, \{ r_i \}_{i \in \Phi_1})$ be a marked point process where $\Phi_1$ is a homogeneous Poisson point process and the marks are independent and uniformly distributed on $(0, R_c)$, with $0< R_c < +\infty$. For $\lambda_2 \in L^\infty(\Rd)$ and $x \in \Rd$, let $\Phi^{x}_2$ be the heterogeneous Poisson point process having intensity $\lambda_2( \cdot - x)$. Then, we define
\begin{align}\label{matern}
\Phi := \bigcup_{z_i \in \Phi_1} \Phi^{z_i}_2 \cap B_{r_i}(z_i).
\end{align}

\smallskip

\item[$(c.2)$]  Strauss process $\Phi$ on $\Rd$  with parameters $\alpha > 0$, $\beta \in [0,1]$ and interaction distance $r_c > 0$ \cite[Example 7.1(c)]{Daley.Jones.book},
\cite{Kelly.Ripley76}. 
For each bounded Borel set $B \subset \Rd$, we define 
\begin{align}\label{strauss}
\P (  \#(\Phi \cap B)= n ) = Z_B^{-1}  \frac{\alpha^n}{n!} \int_{B \times \cdots \times B} \beta^{R(\{x_1, \cdots , x_n \})} dx_1 \, \cdots \, dx_n ,
\end{align}
with 
\begin{align}\label{interaction.strauss}
R(\{x_1, \cdots , x_n \}):= \frac 1 2 \sum_{i,j \atop i \neq j}^n \1_{[0, r_c]}(|x_i - x_j|),
\end{align}
and 
$$
Z_B =  \sum_{n=0}^{+\infty}\frac{\alpha^n}{n!} \int_{B \times \cdots \times B} \beta^{R(\{x_1, \cdots , x_n \})} dx_1 \, \cdots \, dx_n.
$$
This probability measure is well-defined in the repulsive case $\beta \in [0,1)$, while it requires further assumptions in the attractive case $\beta \geq 1$. For $\beta =0$, we remark that $\Phi$ is the hard-core process with radius $r_c$ and intensity $\alpha$ (see also  \cite[Example 5.3(c)]{Daley.Jones.book} and \cite{Kelly.Ripley76}).
We remark that this process is the same as the macrocanonical Gibbs ensemble of Statistical Physics at temperature $T=1$ for the pair-interaction potential
\begin{equation*}
\psi(r) = 
\begin{cases}
-\log\beta \ \ \text{if  $r \leq r_c$}\\
0 \ \ \text{if $r > r_c$}
\end{cases}  
\end{equation*}
and the chemical potential $\mu= \log\alpha$.
\end{itemize}  

\smallskip

For each one of the previous point processes $\Phi$, we let the marginal $\P_{\mathcal{R}}$ be any probability measure satisfying \eqref{correlation.radii.0}
and \eqref{correlation.radii.2}.
\end{itemize}

\section{Proof of Theorem \ref{t.main}}\label{proof.main}
As already discussed in the introduction, our strategy is to adapt the method of \cite{CioranescuMurat} and to show that, in spite of the unboundedness of the radii of the holes in $H^\eps$, we may almost surely construct a sequence of suitable oscillating test function. The crucial result of our paper is indeed the following:

\begin{lem}\label{l.test.periodic}
 Let  $H^\varepsilon= H^\varepsilon(\omega)$ be as in Theorem \ref{t.main}. Then, for $\P$-almost every $\omega \in \Omega$, there exists a sequence $\{ w_\varepsilon( \omega, \cdot) \}_{\varepsilon > 0} \subset H^1(D)$  which satisfies
  \begin{itemize}
 \item[(H1)] For every $\varepsilon > 0$,  $w_\varepsilon(\omega, \cdot) = 0$ in $H^\eps(\omega)$;
 \item[(H2)] $w_\varepsilon( \omega, \cdot) \rightharpoonup 1$ in $H^1(D)$ for $\eps \downarrow 0^+$;
 \item[(H3)] For every sequence $v_\varepsilon \rightharpoonup v$ in $H^1_0(D)$ such that $v_\varepsilon \in H^1_0( D^\varepsilon)$ it holds that
  \begin{align}\label{(H5)}
 ( -\Delta w_\varepsilon(\omega, \cdot), v_\varepsilon )_{H^{-1}, H^{1}_0} \rightarrow C_0 \langle \rho^{d-2} \rangle \int_D v
 \end{align}
for $\eps \downarrow 0^+$ and where $C_0$ defined as in Theorem \ref{t.main}.
 \end{itemize}
\end{lem}

By relying on the previous lemma, the proof of Theorem \ref{t.main} follows exactly as in \cite{CioranescuMurat}:
\begin{proof}[Proof of Theorem \ref{t.main}]
Let $\omega \in \Omega$ be fixed, and let it belong to the full-probability set $\Omega' \subset \Omega$ made of configurations for which, according to  Lemma \ref{l.test.periodic}, the functions  $\{ w_\eps \}_{\eps > 0}:= \{w_\eps(\omega, \cdot)\}_{\eps>0}$ exist and satisfy hypothesis (H1), (H2) and (H3). 

\medskip

Since $u_\eps \in H^1_0(D^\eps)$, we may test equation \eqref{laplace.eps} with $u_\eps$ itself and get by the standard energy estimate
\begin{align*}
\| u_\eps \|_{H^1} \leq  C \|f\|_{H^{-1}},
\end{align*} 
with a constant $C$ that depends only on the domain $D$.
By weak-compactness of $H^1_0(D)$, we infer that, up to a subsequence which may depend on $\omega$,
\begin{align}\label{convergence.sols}
u_\eps \rightharpoonup \uh \ \ \ \ \text{in $H^1_0(D)$ for $\eps \downarrow 0^+$.}
\end{align}
We show that $\uh \in H^1_0(D)$ is the solution of \eqref{laplace.hom}; by uniqueness, this extends the weak convergence of the solutions $u_\eps$ to the continuum limit $\eps \downarrow 0$ and concludes the proof of the theorem.

\medskip 
 
To prove that $\uh$ solves \eqref{laplace.hom}, let us fix any function $\phi \in C^\infty_0(D)$. Since (H1) yields $w_\eps \phi \in H^1_0(D^\eps)$, we can test the equation \eqref{laplace.eps} with $w_\eps \phi$ and obtain
\begin{align}\label{weak.formulation.1}
\int \nabla( w_\eps \phi) \cdot \nabla u_\eps = (f, w_\eps \phi)_{H^{-1}, H^1_0}.
\end{align}
By (H2), the right-hand side above converges to
\begin{align}\label{weak.formulation.2}
 (f, w_\eps \phi)_{H^{-1}, H^1_0} \to (f, \phi)_{H^{-1}, H^1_0}.
\end{align}
We now use the product-rule and an integration by parts to rewrite the left-hand side in \eqref{weak.formulation.1} as 
\begin{align*}
\int \nabla( w_\eps \phi) \cdot \nabla u_\eps &= \int  \phi \nabla w_\eps \cdot \nabla u_\eps + \int w_\eps \nabla \phi  \cdot \nabla u_\eps  \\
& = ( -\Delta w_\eps, \phi u_\eps )_{H^{-1}, H^1_0} - \int u_\eps \nabla w_\eps \cdot \nabla \phi + \int w_\eps \nabla \phi \cdot \nabla u_\eps.
\end{align*}
Since by (H2) of Lemma \ref{l.test.periodic} and \eqref{convergence.sols}, both $u_\eps$ and $w_\eps$ converge strongly in $L^2_{\mathrm{loc}}(D)$, the last two terms on the right-hand side above converge to $\int  \nabla \phi \cdot \nabla \uh $. Furthermore, by \eqref{convergence.sols} and the assumption on $\phi$, we apply hypothesis (H3) of Lemma \ref{l.test.periodic} to the first term on the right-hand side above and conclude that
\begin{align*}
\int \nabla( w_\eps \phi) \cdot \nabla u_\eps \to C_0 \int \phi \uh + \int  \nabla \phi \cdot \nabla \uh.
\end{align*}
This, together with  \eqref{weak.formulation.1}, \eqref{weak.formulation.2} and the arbitrariness of $\phi \in C^\infty_0(D)$, yields that $\uh$ weakly solves \eqref{laplace.hom}. The proof of Theorem \ref{t.main} is complete. 
\end{proof}
\section{Existence of the oscillating test functions (Proof of Lemma \ref{l.test.periodic})}\label{section.oscillating}
As already mentioned in Subsection \ref{main.ideas}, we proceed to prove Lemma \ref{l.test.periodic} in two steps: We first give an argument in the simplest case of random holes
$H^\eps$ having periodic centres and i.i.d. radii (cf. example (a) of Section \ref{setting}). In that case, the crucial role played by assumption \eqref{AverRadii} on the random geometry of the set $H^\eps$ becomes clear.
We then generalize this argument to an arbitrary process $(\Phi, \mathcal{R})$ that satisfies the assumptions of Theorem \ref{t.main}. We observe that, as it becomes apparent in the proofs of this section, the full-probability set of realizations for which the statement of Lemma \ref{t.main} holds true is selected by countable repeated applications of Strong Laws of Large Numbers-type of results. The final set $\Omega' \subset \Omega$ in which we prove the existence of the oscillating test functions is thus a countable intersection of full-probability sets and remains of full probability.

\bigskip

Before giving the proof of Lemma \ref{l.test.periodic}, we fix the following notation: For any two open sets $A \subset B\subset \R^d$, we define
\begin{align}\label{capacity}
\capacity(A, B) := \inf \biggl\{ \int |\nabla v|^2 \, : \, v \in C^\infty_0(B), \, v \geq \1_{A} \biggr\}.
\end{align}

\smallskip

For a point process $\Phi$ on $\R^d$ and any bounded set $E \subset \R^d$, we define the random variables
\begin{align}
	\label{Number.function}
	\Phi(E)&:= \Phi \cap E, && \Phi^\eps(E):= \Phi \cap \left(\frac 1 \eps E \right), \\
  N(E) &:= \# (\Phi (E)), && N^{\eps}(E) := \# (\Phi^\eps(E)).
\end{align}
For $\delta > 0$, we denote by  $\Phi_\delta$ a thinning for the process $\Phi$ obtained as 
\begin{align}\label{thinned.process}
\Phi_\delta(\omega):= \{ x \in \Phi(\omega) \,  \colon \, \min_{y \in \Phi(\omega), \atop y \neq x} | x- y| \geq \delta \},
\end{align}
i.e. the points of $\Phi(\omega)$ whose minimal distance from the other points is at least $\delta$. Given the process $\Phi_\delta$, we set $\Phi_\delta(E)$, $ \Phi_\delta^\eps(E)$, $N_\delta(E)$ and $N_\delta^\eps(E)$ for the analogues for  $\Phi_\delta$ of the random variables defined in \eqref{Number.function}.

\smallskip

For a fixed $M>0$, we define the truncated marks
\begin{align}
	\label{truncated.radii}
\mathcal{R}^M:= \{ \rho_{j,M} \}_{z_j \in \Phi}, \ \ \ \ \rho_{j,M} := \rho_j \wedge M.
\end{align}

Furthermore, throughout the proofs, we write 
\begin{align}
	a \lesssim b 
\end{align}
whenever $a \leq C b$ for a constant $C=C(d)$ depending only on the dimension $d$. 
\medskip

Finally, we remark that, under the assumptions of the process $(\Phi, \mathcal{R})$ in Theorem \ref{t.main},
the process $(\Phi, \{  \rho^{d-2}\}_{z_i\in \Zd})$ satisfies the assumptions of Section \ref{aux}, and we therefore may apply
all the results stated in that section.

\subsection{Case (a): Periodic centres}\label{periodic}
In this setting, the holes $H^\varepsilon$ are generated by $\Phi= \Zd$ and a collection of i.i.d. random variables $\{ \rho_i \}_{i \in \Zd}$ satisfying \eqref{AverRadii}.
It is immediate to check that the marked process 

\medskip

Since the centres of the holes are periodically distributed, the only challenge in the construction of the functions $w_\eps$ of Lemma \ref{l.test.periodic} is due to the random variables $\{\rho_i\}_{i\in \Zd}$ which 
might generate very large holes under the mere condition \eqref{AverRadii}. In fact, in \cite{CioranescuMurat} the construction of $w_\eps$ relies on the assumption that each hole $B_{\aeps}(\eps z_i)$, $z_i \in \Zd$, is strictly contained in the concentric cube of size $\eps$; this allows to explicitly construct $w^\eps$ by locally solving a PDE on each of these cubes. In our case, the sole assumption \eqref{AverRadii} does not exclude that there are big holes which overlap and where the previous construction breaks down. The main auxiliary result on which Lemma \ref{l.test.periodic} for the periodic case (a) relies is the following Lemma \ref{l.geometry.periodic} on the asymptotic geometry of the set $H^\eps$. Roughly speaking, this lemma ensures that $H^\eps$ may be almost surely partitioned into two subsets, a ``good'' and a ``bad'' set of holes which we denote by $H^\eps_g$ and $H^\eps_b$, respectively. The set $H^\eps_g$ contains most of the holes of $H^\eps$ and is made of small balls where the construction of $w_\eps$ may be carried out similarly to \cite{CioranescuMurat}. The remaining holes, some of which overlap with full probability, are all included in $H^\eps_b$.  This set is well separated from $H^\eps_g$ and small with respect to the macroscopic size of the domain $D$: We may indeed enclose  $H^\eps_b$ into a set $D^\eps_b \subset D$ which is still separated from $H^\eps_g$ and such that the harmonic capacity of $H^\eps_b$ with respect to this ``safety layer'' $D^\eps_b$ vanishes in the limit $\eps \downarrow 0^+$. This allows us to implicitly define $w^\eps$ in $D^\eps_b$ as the capacitary function of $H^\eps_b$ in $D^\eps_b$; this choice ensures that the $H^1$-norm of $w_\eps$ on $D^\eps_b$ converges to zero. Hence, in the verification of (H2) and (H3) of Lemma \ref{l.test.periodic}, we only need to focus on the construction of $w_\eps$ on $D\backslash D^\eps_b$.

\begin{lem}\label{l.geometry.periodic} Let $\delta \in (0,\frac{2}{d-2})$ be fixed. Then, there exists $\eps_0= \eps_0(\delta) > 0$ such that for $\P$-almost every $\omega \in \Omega$ and for all $\eps \leq \eps_0$ there exist $H^\varepsilon_{g}(\omega), H^\varepsilon_{b}(\omega), D^\eps_b(\omega) \subset \R^d$ such that
\begin{align} \label{distance.good.bad}
H^\varepsilon(\omega) = H^\varepsilon_{g}(\omega) &\cup H^\varepsilon_{b}(\omega), \qquad  H^\varepsilon_{b}(\omega) \subset D^\eps_b(\omega), \\
&\dist\bigl( H^\varepsilon_g(\omega),  D^\eps_b(\omega) \bigr)\geq \frac{\varepsilon}{2},
\end{align}
where
\begin{align}\label{bad.set}
\lim_{\varepsilon \downarrow 0^+} \capacity\left( H^\varepsilon_b(\omega),D_b^\eps(\omega)\right) = 0.
\end{align}
Moreover, $H^\varepsilon_g(\omega)$ may be written as the following union of  disjoint balls centred in $n^\eps(\omega) \subset \Z^d \cap \frac{1}{\varepsilon}D$:
\begin{align}
H^\varepsilon_g(\omega): = &\bigcup_{ z_j \in n^\varepsilon} B_{\varepsilon^{\frac{d}{d-2}} \rho_j}( \varepsilon z_j ), \label{good.set.0}\\
  \varepsilon^{\frac{d}{d-2}} \rho_j \leq \varepsilon^{1+ \delta}< \frac{\eps}{2},\ & \ \ \lim_{\eps \downarrow 0^+}\eps^{d}\#(n^\varepsilon) = |D|.\label{good.set}
\end{align}
\end{lem}

\begin{proof}[Proof of Lemma \ref{l.geometry.periodic}]
The partition of the set $H^\eps(\omega)$ in the statement of the lemma clearly depends on the realization $\omega \in \Omega$; in the sake of a leaner notation, though,  in the rest of the proof we omit the argument $\omega$ and write $H^\eps$, $H^\eps_b$, $H^\eps_g$ instead of $H^\eps(\omega)$, $H^\eps_b(\omega)$, $H^\eps_g(\omega)$. For each $z_i \in \Z^d$, we denote by $Q^\eps_i$ the cube of length $\eps$ centered at $\eps z_i$. 

\medskip

We begin by constructing the set $H^\eps_b$ and its ``safety layer'' $D_b^\varepsilon$. We first  include in $H^\eps_b$ the particles which are large compared to the size of the cubes $Q_i^\eps$: For $\delta$ as in the statement of the lemma,we consider the subset of $\Zd$ given by
\begin{align}\label{J.set}
J_\eps^b := \Bigl\{z_i \in \Z^d \cap \frac{1}{\eps} D \colon  \eps^\frac{d}{d-2} \rho_j \geq \eps^{1+\delta}\Bigr\},
\end{align}
and the corresponding union of balls
$$
\tilde{H}^\eps_b := \bigcup_{ z_j \in J^\eps_b} B_{2\eps^\frac{d}{d-2} \rho_j}(\eps z_j).
$$
We now extend $J^\eps_b$ by including the centres of the balls for which, independently from the size of their radius, the corresponding cell $Q^\eps_i$ intersects $\tilde{H}^\eps_b$: We define
\begin{align}\label{definition.index}
\tilde{I}^\eps_b := \bigl\{ z_i \in \Z^d   \colon Q^\eps_i \cap  \tilde{H}^\eps_b \neq \emptyset \bigr\} \supset J^\eps_b, \quad I^\eps_b := \tilde{I}^\eps_b \cap \frac{1}{\eps} D.
\end{align}
We finally set
\begin{align}\label{definition.bad}
{H}^\eps_b := \bigcup_{z_j \in I^\eps_b} B_{\eps^\frac{d}{d-2} \rho_j}(\eps z_j), \ \ \  D_b^\eps	 := \bigcup_{z_j \in \tilde{I}^\eps_b} Q^\eps_j.
\end{align}
By \eqref{definition.bad} it is immediate that $H^\eps_b \subset D_b^\eps$. To show \eqref{bad.set} we first argue that provided  $\eps \leq \eps_0(\delta)$, with  $\eps_0(\delta)$ such that $2\eps_0^{1+\delta} \leq \eps_0$, for every  $z_j \in I^\eps_b$ it holds
\begin{align}\label{geo.per.1}
B_{2\eps^{\frac{d}{d-2}}}(\eps z_j)  \subset D_b^\eps.
\end{align}
Indeed, since by definition $\tilde H^\eps_b \subset D_b^\eps$, if $z_j \in J^\eps_b$, then \eqref{geo.per.1} follows immediately. If, otherwise,
$z_j \in I_b^\eps \setminus J^\eps_b$, then the assumption  $\eps \leq \eps_0$ implies $B_{2\eps^{\frac{d}{d-2}}}(\eps z_j) \subset Q_j^\eps \subset D^\eps_b$. 
By the subadditivity of the capacity (see definition \eqref{capacity}) we estimate
\begin{align}
\begin{aligned}
	\mathrm{Cap} \bigl(( H^\varepsilon_b(\omega),D_b^\eps(\omega)\bigr) 
	& \stackrel{\eqref{definition.bad}}	{\leq} \sum_{j \in I^\eps_b} \mathrm{Cap} \Bigl( B_{\eps^{\frac{d}{d-2}} \rho_j}(x_j),D_b^\eps(\omega)\Bigr) \\
	& \stackrel{\eqref{geo.per.1}}{\leq} \sum_{j \in I^\eps_b} \mathrm{Cap} \Bigl( B_{\aeps \rho_j}(x_j),B_{2 \eps^{\frac{d}{d-2}} \rho_j}(x_j)\Bigr) \\
	& ~~ \lesssim \sum_{j \in I^\eps_b} \eps^d \rho_j^{d-2}.
\end{aligned}
\end{align}
To conclude \eqref{bad.set}, it remains to show that the right-hand side above vanishes almost surely in the limit $\eps \downarrow 0^+$. This follows
from Lemma \ref{l.LLN.index.set} for the process $(\Zd, \{\rho^{d-2}_i \}_{z_i \in \Zd})$ provided
\begin{align}\label{geo.per.2}
\lim_{\eps \downarrow 0^+} \eps^d \# (I^\eps_b) = 0 .
\end{align}

\smallskip

To show \eqref{geo.per.2}, we first bound by \eqref{definition.index} 
$$
\eps^d \# (I^\eps_b) \leq \eps^d \# (J^\eps_b)  + \sum_{z_i \in I_b^\eps \backslash J^\eps_b} |Q_i^\eps|.
$$
We note that by \eqref{definition.index} and \eqref{J.set}, there exists a constant $c=c(d)$ such that, provided $\eps \leq \eps_0(d)$
(with $\eps_0(d)$ possibly smaller than the one above), for any cube $Q^\eps_i$ with $z_i \in I^\eps_b$, there exists $z_j \in J^\eps_b$ such that
\begin{align}
 Q^\eps_i \subset B_{2c\aeps\rho_j}(\eps z_j).
\end{align}
Since all the cubes $Q^\eps_i$ are (essentially) disjoint, we use the previous inclusion and the definition of $\tilde H^\eps_b$ to bound
\begin{align*}
{ \eps^d \# (I^\eps_b)} \lesssim \eps^d \# (J^\eps_b)  + |\tilde H^\eps_b | &\lesssim \eps^d \# (J^\eps_b)  + \sum_{z_j \in J^\eps_b} \Bigl(\aeps \rho_j\Bigr)^{d}\\
&\lesssim  \eps^d \# (J^\eps_b)  +\Bigl(\aeps \max_{z_j \in \frac{1}{\eps}D \cap \Z^d} \rho_j \Bigr)^2\eps^d\sum_{z_j \in J^\eps_b}  \rho_j^{d-2}.
\end{align*}
By Lemma \ref{SLLN.general.pp}, we have almost surely 
\begin{align}\label{estimate.max}
\limsup_{\eps \downarrow 0^+} \aeps \max_{z_j \in \frac{1}{\eps}D \cap \Z^d} \rho_j 
\leq \lim_{\eps \downarrow 0^+} \bigl( \eps^d \sum_{z_j \in \frac{1}{\eps}D \cap \Z^d}\rho_j^{d-2}\bigr)^{\frac{1}{d-2}}
= \langle \rho^{d-2} \rangle^{\frac{1}{d-2}},
\end{align}
and thus estimate for $\eps$ small enough (this time depending on $\omega$)
\begin{align}\label{geo.per.3}
\eps^d \# (I^\eps_b) \lesssim \eps^d \# (J^\eps_b)  + \langle \rho^{d-2} \rangle^{\frac{2}{d-2}}\sum_{z_j \in J^\eps_b} \Bigl(\aeps \rho_j\Bigr)^{d-2}.
\end{align}
The first term on right-hand side above tends to zero thanks to
\begin{align*}
\eps^d \#( J^\eps_b) = \eps^d  \sum_{z_j \in \frac{1}{\eps}D \cap \Z^d}\1_{\aeps \rho_j \geq\eps^{1+\delta}} 
\leq \eps^{2-\delta(d-2)} \eps^d \sum_{ \frac{1}{\eps}D \cap \Z^d}\rho_j^{d-2}
\end{align*}
and the choice $\delta< \frac{2}{d-2}$ together with the right-hand side of \eqref{estimate.max}. By this estimate and Lemma \ref{l.LLN.index.set}, also the second term on the right-hand side of \eqref{geo.per.3} vanishes almost surely in the limit $\eps \downarrow 0^+$. We thus established \eqref{geo.per.2} and therefore also \eqref{bad.set}.

\medskip

We now define $H^\varepsilon_{g}:= H^\varepsilon \backslash H^\varepsilon_{b}$, which allows to write $H^\varepsilon_{g}$ as in \eqref{good.set.0}
with $n_\eps = (\Zd \cap \frac 1 \eps D) \backslash  I^\eps_b$.  The first property in \eqref{good.set} is immediately implied by $J_b^\eps \subset I^\eps_b$ and \eqref{J.set}. The second property in \eqref{good.set} follows from \eqref{geo.per.2}.

\smallskip

 It remains to prove the last inequality in \eqref{distance.good.bad}: By the definition of $H^\eps_g$ itself,
 if $B_{\aeps \rho_j}(\eps z_j) \subset H^\eps_g$, then $\aeps\rho_j \leq \eps^{1+\delta}$. We choose $\eps \leq \eps_0(\delta)$ as in \eqref{geo.per.1}, such that $B_{\aeps \rho_j}(\eps z_j) \subset Q^\eps_j$
 and
\begin{align*}
\frac\eps 2 \leq \text{dist} \Bigl( B_{\aeps \rho_j}(\eps z_j) , \partial Q^\eps_j\Bigr) 
\stackrel{\eqref{definition.index}}{\leq} \text{dist}\Bigl(  B_{\aeps \rho_j}(\eps z_j) , D^\varepsilon_{b}\Bigr). 
\end{align*}
This concludes the proof of Lemma \ref{l.geometry.periodic}.
\end{proof}

\bigskip

\begin{proof}[Proof of Lemma \ref{l.test.periodic}, case (a).] Let us fix $\delta$ and $\eps_0(\delta)$ as  in  the statement of Lemma \ref{l.geometry.periodic}. Then, we know that we may fix $\P$-almost any event  $\omega \in \Omega$ such that we find $H_b^\varepsilon(\omega)$, $H^\varepsilon_g(\omega)$ and $D^\eps_b(\omega)$ as in Lemma \ref{l.geometry.periodic}. Also in this proof, to keep the notation leaner,  we omit the argument $\omega$ in the oscillating test functions and in the set of holes and write, for instance, $w_\varepsilon$, $H^\eps$ instead of $w_\varepsilon(\omega, \cdot)$ and $H^\eps(\omega)$.

\medskip

\emph{ Step 1.} We begin by a reduction argument: We claim that we may separately treat the two regions $D_b^\eps$ and $D \backslash D_b^\eps$, 
which contain $H_b^\eps$  and $H_g^\eps$ respectively,  and give an explicit construction for $w^\eps$ only in the set $D \backslash D_b^\eps$ . We indeed claim that, for $\P$-almost every $\omega \in \Omega$ we may set $w^\eps = w^\eps_1 \wedge w^\eps_2$ with
 $w_1, w_2 \in H^1(D)$ and such that
\begin{align}
& {} & w^\eps_1 &\equiv 1 \quad \text{in  $D \setminus D_b^\eps$,}  & w^{\eps}_1 &= 0 \quad \text{in $H^\eps_b$,} \label{splitting.bad} \\
0 &\leq w_2^\eps \leq 1, & w^\eps_2 &\equiv 1 \quad \text{in $D^\eps_b$,}  & w^\eps_2 &= 0 \quad \text{ in $H^\eps_g$,} \label{splitting.good}
\end{align}
with, in addition,
\begin{align}\label{reduction.bad}
 w^\eps_1 \to 1 \ \ \ \ \text{ in $H^1(D)$}.
\end{align}
If true, this decomposition for $w^\eps$ yields that (H1) is satisfied and, by \eqref{reduction.bad}, that (H2) needs to be argued only for  the sequence $\{w^\eps_2 \}_{\eps > 0}$. Finally,  since $\nabla w^{\eps}_1$ and $\nabla w^\eps_2$ have disjoint support, for any sequence $\{ v_\eps\}_{\eps>0} \subset H^1(D)$ as in (H3) we have that 
\begin{align}
( -\Delta w^\eps, v_\eps )_{H^{-1}, H^1_0} = \int \nabla w^\eps_1 \cdot \nabla v_\eps + \int \nabla w^\eps_2 \cdot \nabla v_\eps,
\end{align}
and, by \eqref{reduction.bad}, that the first term on the right-hand side vanishes in the limit $\eps \downarrow 0^+$. Since by an integration by parts, the second term on the right-hand side may be rewritten as $( -\Delta w^\eps_2, v_\eps )_{H^{-1}, H^1_0}$, we deduce that with the previous decomposition we may verify (H3) only for the measures $\{-\Delta w^\eps_2\}_{\eps >0}$.

\medskip

\emph{Step 2.} Construction of $w^\eps_1$ and $w^\varepsilon_2$. We begin with $w^\eps_1$: Thanks to  \eqref{distance.good.bad} of Lemma \ref{l.geometry.periodic} for $H^\eps_b, H^\eps_g$ and $D^\eps_b$, together with \eqref{capacity} and \eqref{bad.set}, for every $\eps \leq \eps_0$ there exists a function $\tilde w^\eps_1 \in H^1_0( D_b^\eps)$, such that $\tilde w^\eps_1 = 1$ in $H^\eps_b$, which satisfies
$$
\int_{D_b^\eps} |\nabla \tilde w^\eps _1|^2 \leq 2 \, \text{Cap}( H^\eps_b, D^\eps_b).
$$
If we now set  $w^\varepsilon_1= 1 - \tilde w^\eps_1$, and trivially extend $w_1^\eps$ by $1$ outside $D^\eps_b$, we immediately have that \eqref{splitting.bad} for $w^\eps_1$ is satisfied. In addition, thanks to \eqref{bad.set} and our choice of $\tilde w^\eps_1$, also \eqref{reduction.bad} follows. 

\medskip

We now turn to the construction of $w^\eps_2$: By the properties of $H^\eps_g$, $H^\eps_b$ and $D^\eps_b$ of Lemma \ref{l.geometry.periodic}, the set $D\backslash D^\eps_b$ contains only the holes of $H^\eps_g$, which are all disjoint balls, each strictly contained in the concentric cube $Q^\eps_i$ of size $\eps$.
We define $w^\eps_2 \equiv 1$ on $D_b^\eps$, and explicitly construct $w^\eps_2$ on $D \backslash D^\eps_b$ as done in \cite{CioranescuMurat}:
For each $z_i \in n^\eps$, with $n^\eps$ defined in the statement of Lemma \ref{l.geometry.periodic}, we write $T_i^\eps = B_{\eps^{\frac{d}{d-2}} \rho_i} (\eps z_i)$ and $B_i =  B_{\frac \eps 2}(\eps z_i)$ and define
\begin{align}\label{definition.oscillating.2}
 w^\varepsilon_2 = 1 - \sum_{z_i \in n^\eps} w^{\varepsilon, i}_2,
\end{align}
with each $w^{\varepsilon, i}_2$ solving
\begin{align}
	\label{cell.problem.0}
\begin{cases}
	-\Delta w^{\varepsilon, i}_2 = 0 \quad \text{in} ~ B_i \setminus T_i\\
	w^{\varepsilon, i}_2 = 1 \quad \text{in} ~ T_i\\
	w^{\varepsilon, i}_2 = 0 \quad \text{in} ~ D \setminus B_i.
\end{cases}
\end{align}

\medskip

Since by Lemma \ref{l.geometry.periodic} all the balls $B_i$ are disjoint and contained in $D \backslash D^\eps_b$, definitions \eqref{definition.oscillating.2} and \eqref{cell.problem.0} yield that $w^\eps_2$ satisfies \eqref{splitting.good} of Step 1. We thus constructed $w^\eps_1, w^\eps_2$ satisfying \eqref{splitting.good} and \eqref{reduction.bad} of Step 1. We conclude this step by remarking that definition \eqref{cell.problem.0} also implies that 
\begin{align}\label{oscillating.2}
  w^{\varepsilon,i}_2 = 1 - \mbox{argmin} \bigl\{ \mbox{Cap}(T^\eps_i,B^\eps_i) \bigr\},
\end{align}
and that each $w^{\eps,i}_2$ may be written explicitly as
\begin{align}
	\label{cell.problem}
\begin{cases}
	w^{\varepsilon, i}_2(x) &= \frac{|x-  \eps z_i|^{-(d-2)}- (\frac \eps 2)^{-(d-2)}}{\eps^{-d} \rho_i^{-(d-2)} -(\frac \eps 2)^{-(d-2)}} \quad \text{in} ~ B_i \setminus T_i \\
	w^{\varepsilon, i}_2 &= 1 \quad \text{in} ~ T_i\\
	w^{\varepsilon, i}_2 &= 0 \quad \text{in} ~ D \setminus B_i.
\end{cases}
\end{align}

\medskip

\emph{Step 3. } Equipped with $w^\eps_1$ and $w^\eps_2$ constructed above, we show that $w^\eps = w^\eps_1 \wedge w_2^\eps$ satisfies properties (H1)-(H3).
As already discussed in Step 1, it suffices to prove that $\{ w^\eps_2\}_{\eps> 0}$ satisfies (H2) and (H3).

We begin with (H2): By \eqref{definition.oscillating.2}, \eqref{cell.problem} and \eqref{good.set} of Lemma
\ref{l.geometry.periodic}, a direct calculation leads to
\begin{align}
	\label{capacity.good.periodic}
	\|\nabla w_2^{\eps} \|^2_{L^2(D)} & \lesssim \eps^{d}\sum_{z_i \in n^\eps} \rho_i^{d-2}\leq \eps^{d}\sum_{z_i \in \Zd\cap \frac 1 \eps D} \rho_i^{d-2}.
\end{align}
By Lemma \ref{SLLN.general.pp} applied to the right hand side, we infer that, almost surely,
\begin{align}\label{capacity.good.periodic2}
\limsup_{\eps \downarrow 0^+}\|\nabla w^\eps_2 \|^2_{L^2(D)} \leq C.
\end{align}
In addition, since $1 - w^\eps_2 = 0$ in $\Rd \backslash \bigl( \bigcup_{z_i \in n^\eps}B_i \bigr)$, and the balls $\{B_i \}_{z_i \in n^\eps}$ are essentially disjoint, by Poincar\'e's inequality we obtain also
\begin{align}
\| 1- w^\eps_2 \|_{L^2(D)}^2 \leq \sum_{z_i \in n^\eps} \| 1- w^\eps_2 \|_{L^2(B_i)}^2 \lesssim \eps^2 \sum_{z_i \in n^\eps} \|\nabla w^\eps_2 \|_{L^2(B_i)}^2.
\end{align}
This, together with \eqref{capacity.good.periodic} and \eqref{capacity.good.periodic2}, yields that almost surely $w^\eps_2 \wto 1$ in $H^1(D)$ when $\eps \downarrow 0^+$. We thus established  (H2).

\smallskip

To prove (H3) for $w^\eps_2$, we first use  \eqref{definition.oscillating.2} and \eqref{cell.problem} to decompose
\begin{align}\label{decomposition.laplacians}
-\Delta w^\eps_2 = \sum_{i=1}^{n^\eps} (\mu^{\eps,i}- \gamma^{\eps,i}), \ \ \ \mu^{\eps,i}= -\partial_\nu w^{\eps,i}_2 \, \delta_{\partial B_i}, \  \gamma^{\eps,i}=-\partial_\nu w^{\eps,i}_2 \, \delta_{\partial T_i},
\end{align}
with $\nu$ denoting the outer normal and $\delta_{\partial B^i}$ and $\delta_{\partial T^i}$ being the $(d-1)$-dimensional Hausdorff measure restricted to $\partial B_i$
and $\partial T_i$, respectively. We start by remarking (see (H5)' of \cite{CioranescuMurat}) that, since in (H3) the functions $v_\eps$ are always assumed to be vanishing
on each $T_i$, we only need to focus on the convergence (H3) for the sequence of measures 
\begin{align}\label{measure.mu}
\mu^{\eps}:=  -\sum_{i \in n^\eps} \partial_\nu w^{\eps,i}_2 \, \delta_{\partial B_i} : = \sum_{i \in n^\eps} \mu^{\eps,i}.
\end{align}
More precisely, we claim that for every $v_\epsilon \rightharpoonup v$ in $H^1_0(D)$ such that $v_\eps \in H^1_0(D^\eps)$, it holds
\begin{align}\label{convergence.mu}
( \mu^\eps , v_\eps)_{H^{-1}, H^1_0(D)} \rightarrow C_0 \int_D v,
\end{align}
where $C_0:=(d-2)\sigma_d \langle \, \rho^{d-2} \, \rangle$ corresponds to the definition \eqref{CM.term} for the case $\Phi=\Zd$ under consideration.

\smallskip

We begin by arguing that it suffices to prove \eqref{convergence.mu} above for any truncated process $(\Zd, \mathcal{R}^M)$, with $M \in \N$ and $ \mathcal{R}^M$ defined in \eqref{truncated.radii}. From now on, we use the lower index $M$ to distinguish
the objects constructed with the truncated marks $\mathcal{R}^M$ and the ones coming from $\mathcal{R}$. For instance, we denote by $w^\eps_{2,M}$, $\mu^\eps_M$ the analogues
of $w^\eps$ and $\mu^\eps$ introduced above. Note that, for $\mathcal{R}^M$, the constant in \eqref{convergence.mu} reads 
$C_{0,M}= (d-2)\sigma_d \langle \, \rho_M^{d-2} \, \rangle$.

For any $M\in \N$, since  $|C_0 - C_{0,M}| \lesssim \langle\, \rho^{d-2} \1_{\rho \geq M} \, \rangle$,
 we bound by the triangular inequality, an integration by parts and Cauchy-Schwarz inequality  
\begin{align}
\label{convergence.mu1}
\bigg|( - \Delta w^\eps , v_\eps)_{H^{-1}, H^1_0(D)} - C_0 \int_D v\bigg|
&\lesssim \| \nabla (w^\eps_M - w^\eps)\|_{L^2(D)} \| \nabla v_\eps\|_{L^2(D)} \\
& + \bigg|( \mu^\epsilon_M , v_\eps)_{H^{-1}, H^1_0(D)} - C_{0,M} \int_D v \bigg| +\langle\, \rho^{d-2} \1_{\rho \geq M}  \, \rangle \|v\|_{L^1},
\end{align}
By an argument similar to the one in
\eqref{capacity.good.periodic}, we estimate
\begin{align}\label{cutoff.estimate}
\limsup_{\eps \downarrow 0^+} \| \nabla (w^\eps_M - w^\eps)\|_{L^2(D)} \lesssim \langle\,  \rho^{d-2} \1_{\rho \geq M}  \, \rangle,
\end{align}
so that by letting $\eps \downarrow 0^+$ in \eqref{convergence.mu1}, this  and the boundedness of the sequence 
$v_\eps$ in $H^1$ yield
\begin{align*}
\limsup_{\eps \downarrow 0^+}&\bigg|( - \Delta w^\eps , v_\eps)_{H^{-1}, H^1_0(D)} - C_0 \int_D v\bigg| \\
&\lesssim \limsup_{\eps \downarrow 0^+}\bigg|( \mu^\epsilon_M , v_\eps)_{H^{-1}, H^1_0(D)} - C_{0,M} \int_D v \bigg| + \langle\,  \rho^{d-2} \1_{\rho \geq M}  \, \rangle (\|v\|_{L^1} + 1).
\end{align*}
Hence, provided that \eqref{convergence.mu} holds for $\mu^\eps_M$ and any fixed $M\in \N$, we may then send $M \uparrow +\infty$ and establish (H3) by assumption \eqref{AverRadii}.

\bigskip

To argue that almost surely and for every $M \in \N$ the convergence in \eqref{convergence.mu} holds for $\mu^\eps_M$, we follow \cite{CioranescuMurat}.
First, by the definition of $w_{i,M}^\eps$, we compute
\begin{align}
	\mu^\eps_M = \sum_{z_i \in n_\eps \cap \frac 1 \eps D} \frac{2^{d-1}(d-2) (\rho_{i,M})^{d-2}}{1- 2^{d-2}\eps^2 (\rho_{i,M})^{d-2}} \eps \delta_{ \partial B_i}.
\end{align}
Since $\rho_{i,M}\leq M$, to obtain \eqref{convergence.mu} it suffices to prove
\begin{align}\label{strong.convergence.cutoff2}
	\tilde{\mu}^\eps_M := \sum_{z_i \in n_\eps} 2^{d-1} (d-2) \rho_{i,M}^{d-2} \eps \delta_{ B_i} \to C_{0,M}
	\quad \text{strongly in} ~ W^{-1,\infty}(D).
\end{align}
 To show this, we fix $M\in \N$ and split the convergence \eqref{convergence.mu} into the two following steps: If we define 
$$
\eta^\eps_M := \sum_{z_i \in \Z^d \cap \frac 1 \eps D} 2^d (d-2) d \rho_{M,i}^{d-2} \1_{B_i},
$$
then  we argue that
\begin{align}\label{convergence.mu.1}
\tilde \mu^\eps_M- \eta^\eps_M \rightarrow 0 \quad \text{strongly in} ~ W^{-1,\infty}(D),
\end{align}
and
\begin{align}\label{convergence.mu.2}
	 \eta^\eps_M \to C_{0,M} 	\quad \text{strongly in} ~ W^{-1,\infty}(D).
\end{align}

To show \eqref{convergence.mu.1}, we consider the auxiliary problems
\begin{align}
	\label{PDE.q}
\begin{cases}
	-\Delta q^\eps_{i,M} &= 2^d (d-2)d \rho_{i,M}^{d-2} \quad \text{in} ~ B_i^\eps \\
	\frac{\partial q^\eps_{i,M}}{\partial \nu} &= 2^{d-1} (d-2) \rho_{i,M}^{d-2} \eps \quad  \text{on} ~ \partial B_i^\eps,
	\end{cases}
\end{align}
which are in particular satisfied by the functions 
\begin{align}
	q^\eps_{i,M}(x) = 2^{d-1} (d-2) \rho_{i,M}^{d-2} \Bigl(|x-z_i|^2 - \Bigl(\frac \eps 2 \Bigr)^2\Bigr).
\end{align}
As $q^\eps_{i,M} = 0$ on  $ \partial B_i^\eps$, we may extend $q^\eps_{i,M}$ by zero outside of $B_i^\eps$ and estimate
\begin{align}
	\|\nabla q^\eps_{i,M}\|_{L^\infty(B_i)} = 2^{d-1} (d-2)  \rho_{i,M}^{d-2} \eps \lesssim  M^{d-2} \eps.
\end{align}
Using Poincar\'e's inequality, and since the balls $B_i^\eps$ are disjoint, we infer that
\begin{align}
	\label{q_M.to.zero}
	q^\eps_M := \sum_{z_i \in \Z^d \cap \frac 1 \eps D} q^\eps_{i,M} \to 0 \quad \text{in} ~ W^{1,\infty}(\R^d).
\end{align}
We observe that by \eqref{PDE.q}
\begin{align}
	 \eta^\eps_M - {\tilde \mu}^\eps_M  =
	 - \Delta q^\eps_M + \sum_{z_i \in (\Z^d \cap \frac 1 \eps D) \backslash n_\eps} 2^d (d-2) d \rho_{M,i}^{d-2} \1_{B_i} =: - \Delta q^\eps_M + R_\eps^M.
\end{align}
We have by \eqref{q_M.to.zero} that $- \Delta q^\eps_M \to 0$ in $W^{-1,\infty}(\R^d)$. On the other hand, for the term $R_\eps^M$ above, we have that by Lemma \ref{l.LLN.index.set} and \eqref{good.set}
\begin{align}
	\lim_{\eps \downarrow 0^+} \|R^\eps\|_{L^1} 
	\lesssim \lim_{\eps \downarrow 0^+} \eps^d \sum_{z_i \in (\Z^d \cap \frac 1 \eps D) \backslash n_\eps}  \rho_{M,i}^{d-2} = 0,
\end{align}
almost surely. 
Since $R^\eps$ is bounded in $L^\infty$, we also have that $R^\eps \stackrel{*}{\rightharpoonup} 0$ in $L^\infty(D)$ and $R^\eps \to 0$ in $W^{-1,\infty}(D)$. This yields \eqref{convergence.mu.1}.

\smallskip

In order to show \eqref{convergence.mu.2}, we first remark that it suffices to argue that
\begin{align}
\eta^\eps_M \stackrel{*}{\rightharpoonup}  C_{0,M} 	\quad \text{in} ~ L^\infty(D).
\end{align}
Since  the family of functions $\{\eta^\eps_M\}_{\eps> 0}$ is uniformly bounded in $L^\infty(D)$, we identify the $w^*$-limit by testing $\eta^\eps_M$ with any function  $\zeta\in C^1_0({D})$: Indeed, by Lemma \ref{l.w.independent} applied to $(\Zd, \{ \rho^{d-2}_{i,M}\}_{i\in \Zd})$ in the domain $B=D$ we infer almost surely that
$$
( \eta^\eps_M , \zeta)_{H^{-1}, H^1_0(D)} \to C_{0,M} \int_D \zeta.
$$
This establishes \eqref{convergence.mu.2}  and thus concludes the proof for (H3) and for the whole lemma.
\end{proof}

\subsection{Proof of Lemma \ref{l.test.periodic} in the general case}\label{ppp}
Let $(\Phi, \mathcal{R})$ be a marked point process satisfying the assumptions
of Theorem \ref{t.main}.
In contrast with the previous subsection, when the centres of the holes are distributed according to a general point process $\Phi$, there is not a deterministic
positive lower bound for the minimal distance between the points of $\Phi$. This requires some technical changes in the arguments
of Subsection \ref{periodic} but still allows us to obtain a statement on the asymptotic geometry of $H^\eps$ similar to  Lemma \ref{l.geometry.periodic} and to prove Lemma \ref{l.test.periodic}.

\begin{lem}\label{l.geometry.ppp} There exist an $\eps_0=\eps_0(d)$ and a family of random variables $\{r_\eps\}_{\eps>0} \subset \R_+$ 
such that for $\P$-almost every $\omega \in \Omega$
\begin{align}\label{r.eps.correction}
\lim_{\eps \downarrow 0^+}& r_\eps(\omega) = 0,
\end{align}
and for any $\eps \leq \eps_0$ there exist $H^\varepsilon_{g}(\omega), H^\varepsilon_{b}(\omega), D^\eps_b(\omega) \subset \R^d$  such that 
\begin{align} \label{ppp.distance.good.bad}
H^\varepsilon(\omega) = H^\varepsilon_{g}(\omega) \cup H^\varepsilon_{b}(\omega), &\ \ \  H^\varepsilon_{b}(\omega) \subset D^\eps_b(\omega),\notag\\
\dist\bigl(H^\eps_g(\omega), D^\eps_b(\omega)\bigr) &\geq \frac{\eps r_\eps(\omega) }{2},
\end{align}
where
\begin{align}
\lim_{\varepsilon \downarrow 0^+} &\capacity ( H^\varepsilon_b(\omega),D_b^\eps(\omega)) = 0. \label{ppp.bad.set}
\end{align}
Moreover, $H^\varepsilon_g(\omega)$ may be written as the following union of disjoint balls centred in 
$n^\varepsilon(\omega) \subset \Phi(\frac{1}{\varepsilon}D)$:
\begin{align}
H^\varepsilon_g(\omega)&: = \bigcup_{ z_j \in n^\varepsilon} B_{\varepsilon^{\frac{d}{d-2}} \rho_j}( \varepsilon z_j ), \notag \\
\label{good.set.ppp}
 \min_{z_i\neq z_j \in n^\eps}\eps |z_i - z_j|\geq 2 r_\eps \eps, \ \ \ \ \ \varepsilon^{\frac{d}{d-2}} \rho_j &\leq \frac{\eps r_\eps(\omega) }{2}, \ \ \ \ \ \lim_{\eps \downarrow 0^+}\eps^{d}\#(n^\varepsilon) = \langle N(Q) \rangle |D|.
\end{align}

Furthermore, if for $\delta > 0$ the process $\Phi_\delta$ is defined as in \eqref{thinned.process},  then
\begin{align}\label{small.distance.bad}
\lim_{\eps\downarrow 0^+}\eps^d \# \bigl( \bigl\{  z_i \in \Phi^\eps_{2 \delta}(D)(\omega) \, \colon \,\dist( z_i , D^\eps_b ) \leq \delta \eps \bigr\} \bigr)= 0.
\end{align}
\end{lem}

\bigskip

We remark that the lower bounds in \eqref{ppp.distance.good.bad} and \eqref{good.set.ppp} differ from the ones of Lemma \ref{l.geometry.periodic} by the factor $r_\eps$. 
This implies, by \eqref{r.eps.correction}, that the minimal distance between the balls of the ``good'' set $H^\eps_g$ is only $o(\eps)$ for 
$\eps \downarrow 0^+$ and not of order $\eps$ as required in the construction of the functions $\{w_\eps\}_{\eps>0}$ carried out
in the previous subsection. We overcome this technical issue by comparing $w_\eps$ again with
the oscillating test functions $w^\eps_M$ obtained by approximating $H^\eps_g$ by a simpler set $H^{\eps,M}_g$.
Here, $H^{\eps,M}_g$ is obtained not only by truncating the radii at size $M$ as in the previous subsection, but also by considering in $H^{\eps,M}_g$ only the balls whose
centres (in $n^\eps \subset \Phi^\eps(D)$) satisfy \eqref{ppp.distance.good.bad} and \eqref{good.set.ppp} with $M^{-1} \eps$ instead of $r_\eps\eps$. As in the previous
subsection, we show that the sets $H^{\eps,M}$ are a good approximation of $H^\eps$, in the sense that the associated functions $w^\eps_M$ and $w^\eps$ are close in $H^1$. 
This follows from the fact that the centres removed from $H^\eps_g$, which are either too close to each other or to the ``safety layer'' $D^\eps_b$, are few and may be taken
care of by studying the properties of the thinnings $\Phi_{\delta}$ of a process $\Phi$ defined in \eqref{thinned.process}. In fact, the last limit \eqref{small.distance.bad}
of the previous lemma states that the main error in considering the approximate holes $H^{\eps,M}_g$ is given by the points which are too close to each other.

\bigskip

\begin{proof}[Proof of Lemma \ref{l.geometry.ppp}.] 
As in the previous subsection, we suppress the argument $\omega \in \Omega$ for the random sets involved in the argument below. Let us fix an $\alpha \in (0, \frac{2}{d-2})$  and let us define
\begin{align}\label{definition.r}
r_\eps := (\aeps \max_{z_j \in \Phi^\eps(D)} \rho_j)^{\frac{1}{d}} \vee \eps^{\frac{{ \alpha}}{4}}.
\end{align}
With this choice, we prove that the decomposition of $H^\eps$ required by the lemma holds true. We begin by showing that with this definition $r_\eps$ satisfies \eqref{r.eps.correction}: Let
\begin{align}
	F^\eps := \{ z_j \in \Phi^\eps(D) \, | \, \aeps\rho_j \geq \eps\}.
\end{align}
If $F^\eps = \emptyset$, then $r_\eps \leq \eps^{\frac{1}{d}} \vee   \eps^{\frac{\alpha}{4}}$. If otherwise, we estimate
\begin{align}
 \eps^d \max_{z_j \in \Phi^\eps(D)} \rho_j^{d-2} =  \eps^d \max_{z_j \in F^\eps} \rho_j^{d-2} \leq  \eps^d \sum_{z_j \in F^\eps} \rho_j^{d-2}.
\end{align}
To get \eqref{r.eps.correction}, it suffices to show that almost surely the right hand side above tends to zero in the limit $\eps \downarrow 0^+$. By Lemma \ref{l.LLN.index.set}, this holds provided 
\begin{align}
	\label{number.F}
 \eps^d \#(F^\eps) \rightarrow 0 \ \ \ \ \ \eps \downarrow 0^+.
\end{align}
We show this by bounding
\begin{align*}
\eps^d \#(F^\eps) \lesssim \eps^2 \eps^{d} \sum_{z_j \in \Phi^\eps(D)} \rho_j^{d-2},
\end{align*}
and using Lemma \ref{SLLN.general.pp}. We thus established \eqref{r.eps.correction}.

\medskip

Equipped with \eqref{definition.r}, we now set $\eta_\eps= r_\eps \eps$ and begin by constructing the sets $H^\eps_b$ and $D^\eps_b$. As in Lemma \ref{l.geometry.periodic}, we denote by $I^\eps_b$ the set of points in $\Phi^\eps(D)$ which generate the sets $H^\eps_b$ and $D^\eps_b$. We start by requiring that $I^\eps_b$ contains the points in $\Phi^\eps(D)$ whose associated radii are ``too big'', namely
the set
\begin{align}\label{index.set.ppp.1}
J^\eps_b= \Bigl\{ z_j \in \Phi^\eps(D) \colon \aeps\rho_j \geq \frac{\eta_\eps}{2}\Bigr\}.
\end{align}
Similarly to the periodic case, we set
\begin{align}\label{H.tilde.ppp}
 \tilde H^\eps_b := \bigcup_{z_j \in  J^\eps_b} B_{2 \aeps \rho_j}(\eps z_j).
\end{align}
We now include in $I^\eps_b$ also the points in $ \Phi^\eps(D) \backslash J^\eps_b$ which, in spite of having radii below the threshold set in definition \eqref{index.set.ppp.1}, are too close to each other. We indeed define
\begin{align}\label{index.set.ppp.2}
K^{\eps}_{b}:= \Phi^\eps(D) \backslash \left(\Phi^\eps_{2 r_\eps}(D) \cup J^\eps_b\right).
\end{align}
 Finally, we include into $I^\eps_b$ also the set of points which are not in $J^\eps_b \cup K^\eps_b$, but which might are close to $\tilde H^\eps_b$: We denote them by
\begin{align}\label{index.set.ppp.3}
\tilde I^\eps_b := \left \{ z_j \in \Phi^\eps(D) \backslash(J^\eps_b \cup K^\eps_b) \colon  \tilde H^\eps_b  \cap B_{\eta_\eps}(\eps z_j) \neq \emptyset \right\}.
\end{align}
\smallskip
We thus set
\begin{align}\label{definition.bad.index.ppp}
	I^\eps_b& = \tilde I^\eps_b \cup J^\eps_b \cup  K^\eps_b,\\
H^\eps_b:=  \bigcup_{z_j \in I^\eps_b} B_{\aeps \rho_j}(\eps z_j), \ \ \ H^\eps_g&:= H^\eps \setminus H^\eps_b, \ \ \ \ \ D^\eps_b:= \bigcup_{z_j \in I^\eps_b}  B_{2\aeps \rho_j}(\eps z_j). \label{bad.set.ppp}
\end{align}

\medskip

It remains to show that, with $H^\eps_b$, $H^\eps_g$ and $D^\eps_b$ defined as in \eqref{bad.set.ppp}, properties  \eqref{ppp.distance.good.bad}, \eqref{ppp.bad.set}, \eqref{good.set.ppp} and \eqref{small.distance.bad} are satisfied. We start with \eqref{ppp.bad.set}. As in the proof of Lemma \ref{l.geometry.periodic}, the definition of $D^\eps_b$ allows us to estimate by the sub-additivity of the capacity
\begin{align*}
 \text{Cap}(H^\eps_b, D^\eps_b) \lesssim \eps^d\sum_{z_j \in I^\eps_b}\rho_j^{d-2}.
\end{align*}
Thanks to Lemma \ref{l.LLN.index.set}, we conclude that, almost surely, when $\eps \downarrow 0^+$, the right-hand side above vanishes provided that almost surely
\begin{align}
	\label{number.I_b}
 \lim_{\eps\downarrow 0^+}\eps^d \#(I^\eps_b) = 0.
\end{align}
We show this by using definition \eqref{definition.bad.index.ppp} and proving that each of the sets which constitute $I^\eps_b$ satisfies the limit above. We begin with $J^\eps_b$: Definitions \eqref{definition.r} and \eqref{index.set.ppp.1} yield
\begin{align*}
\eps^d \#(J^\eps_b) \lesssim \eps^2 r_\eps^{-(d-2)}  \eps^{d} \sum_{z_j \in \Phi^\eps(D)} \rho_j^{d-2} \leq \eps^{2 - \alpha(d-2)}  \eps^{d} \sum_{z_j \in \Phi^\eps(D)} \rho_j^{d-2}.
\end{align*}
By Lemma \ref{SLLN.general.pp} and the assumption $\alpha < \frac{2}{d-2}$, the right-hand side almost surely vanishes in the limit $\eps \downarrow 0^+$. We thus established
\begin{align}\label{number.J_b}
\lim_{\eps \downarrow 0^+} \eps^d \#(J^\eps_b)=0,
\end{align}
i.e. limit \eqref{number.I_b} for $J^\eps_b$. We now turn to $K^\eps_b$: Let $\{\delta_k \}_{k \in \N}$ be any sequence such that $\delta_k \downarrow 0^+$. Since $r_\eps$ satisfies \eqref{r.eps.correction}, we estimate for any $\delta_k$
\begin{align}
	 \limsup_{\eps \downarrow 0^+} \eps^d \#(K^\eps_b)\stackrel{\eqref{index.set.ppp.2}}{=} \limsup_{\eps \downarrow 0^+} \eps^d \bigl(N^\eps(D) - N^\eps_{r_\eps}(D)\bigr) 
	 \stackrel{\eqref{thinned.process}}{\leq} \limsup_{\eps \downarrow 0^+} \eps^d \bigl(N^\eps(D) - N^\eps_{\delta_k}(D)\bigr).
\end{align}
 We now apply Lemma \ref{SLLN.general.pp} to $\Phi$ and each $\Phi_{\delta_k}$ to deduce that almost surely and for every $\delta_k$
\begin{align}
	\limsup_{\eps \downarrow 0^+} \eps^d \#(K^\eps_b) \leq \langle  N(Q) - N_{\delta_k}(Q) \rangle |D|,
\end{align}
where $Q$ is a unit cube. By sending $\delta_k \downarrow 0^+$, Lemma \ref{SLLN.general.pp} yields 
\begin{align}
	\label{number.K_b}
	\lim_{\eps \downarrow 0^+} \eps^d \#(K^\eps_b) = 0.
\end{align}

\smallskip

To conclude the proof of \eqref{number.I_b}, it remains to show that almost surely also
\begin{align}
	\label{number.tildeI_b}
 \eps^d \#(\tilde I^\eps_b) \rightarrow 0 \ \ \ \ \ \eps \downarrow 0^+.
\end{align}
By definitions \eqref{index.set.ppp.1}, \eqref{index.set.ppp.2} and \eqref{index.set.ppp.3}, for each $z_i \in  \Phi^\eps(D)\backslash (J^\eps_b \cup K^\eps_b)$, we have 
\begin{align}\label{minimal.dist}
 \min_{z_j \in \Phi^\eps(D) \backslash\{ z_i\}}\eps| z_j - z_i| \geq 2\eta_\eps, \qquad  \aeps\rho_i < \frac{\eta_\eps}{2}.
\end{align}
On the one hand, by the first inequality above, the balls $\{ B_{\eta_\eps}( \eps z_i ) \}_{z_i \in \tilde I^\eps_b}$ are all disjoint and satisfy
\begin{align*}
\eps^d \#(\tilde I^\eps_b) &\lesssim \eps^d \sum_{z_i \in \tilde I^\eps_b} \eta_\eps^{-d} |B_{\eta_\eps}(\eps z_i)| =  r_\eps^{-d} \sum_{z_i \in \tilde I^\eps_b} |B_{\eta_\eps}(\eps z_i)| .
\end{align*}
In addition, the inequalities \eqref{minimal.dist}, definitions \eqref{index.set.ppp.1}, \eqref{index.set.ppp.2} and  \eqref{H.tilde.ppp} also imply that  $|\bigcup_{z_i \in \tilde I^\eps_b} B_{\eta_\eps}( \eps z_i )| \lesssim |\tilde H^\eps_b|$. By wrapping up the previous two inequalities, we bound
\begin{align}
	\label{number.tildeI_b.2}
	\eps^d \#(\tilde I^\eps_b) &\lesssim r_\eps^{-d} |\tilde H^\eps_b| \stackrel{\eqref{H.tilde.ppp}}{\lesssim}  r_\eps^{-d} \sum_{z_j \in J_b^\eps} \Bigl(\aeps \rho_j\Bigr)^d \\
	& \lesssim r_\eps^{-d} \biggl(\aeps \max_{z_j\in J^\eps_b} \rho_j\biggr)^2 \sum_{z_j \in  J^\eps_b} \Bigl(\aeps \rho_j\Bigr)^{d-2}.
\end{align}
By definition \eqref{definition.r}, the inequality above reduces to
\begin{align*}
\eps^d \#( \tilde I^\eps_b) \lesssim  \eps^d \sum_{j \in  J^\eps_b} \Bigl(\aeps \rho_j\Bigr)^{d-2}.
\end{align*}
Thanks to \eqref{number.J_b}, we apply Lemma \ref{l.LLN.index.set} and deduce \eqref{number.tildeI_b}. This, together with \eqref{number.J_b} and \eqref{number.K_b}, yields \eqref{number.I_b} and concludes the proof of \eqref{ppp.bad.set}.

\medskip

To show \eqref{ppp.distance.good.bad}, we recall the definitions of $D^\eps_b$, $H^\eps_b$ and $H^\eps_g$ in \eqref{bad.set.ppp} and set $n^\eps:= \Phi^\eps(D)\backslash I^\eps_b$. Since all $z_i \in n^\eps$ satisfy \eqref{minimal.dist} and thus also
\begin{align}
\operatorname{dist}\Bigl( B_{ \aeps \rho_i}(\eps z_i) , \partial B_{\eta_\eps}(\eps z_i)\Bigr) \geq \frac{\eta_\eps}{2},
\end{align}
 by definition \eqref{bad.set.ppp} of $D_b^\eps$, it suffices to show  that for all $z_i \in \Phi^\eps(D)\backslash I^\eps_b$ and all $z_j \in I^\eps_b$ we have
\begin{align}
	\label{ppp.good.bad.intersection}
	 B_{\eta_\eps}(\eps z_i) \cap  B_{2 \aeps \rho_j}(\eps z_j) = \emptyset.
\end{align}
For all $z_j \in J_b^\eps \subset I^\eps_b$, this identity holds by \eqref{index.set.ppp.3} and the definition of $n^\eps$. If $z_j \in I_b^\eps \backslash J_b^\eps$,
then we know that $2 \aeps \rho_j \leq \eta_\eps$ and, by \eqref{minimal.dist} for $z_i$, we obtain \eqref{ppp.good.bad.intersection} also in this case. This establishes  \eqref{ppp.good.bad.intersection} and also \eqref{ppp.distance.good.bad}.

\medskip

Finally, the properties \eqref{good.set.ppp} of the set $H^\eps_g$ are a consequence of \eqref{minimal.dist}, definition \eqref{number.I_b} and \eqref{average.number} of  Lemma \ref{SLLN.general.pp}. 

\medskip

 To show \eqref{small.distance.bad}, we resort to the definition of $D^\eps_b$ to estimate
  \begin{align*}
 \bigl\{  z_i \in \Phi^\eps_{2\delta}(D)(\omega)  &\colon  \mbox{dist}( z_i , D^\eps_b ) \leq \delta \eps \bigr\}\\
 & \begin{aligned}
 	\subset I_b^\eps & \cup \Bigl\{  z_i \in n^\eps(\omega)  \colon  \dist\Bigl( z_i , \bigcup_{z_j \in J^\eps_b} B_{2 \aeps \rho_j}(\eps z_j)  \Bigr) \leq \delta \eps \Bigr\} \\
  &\cup  \Bigl\{  z_i \in n^\eps(\omega) \cap \Phi^\eps_{2\delta}(D)(\omega) \colon  \dist\Bigl( z_i , \bigcup_{z_j \in \tilde I^\eps_b \cup K^\eps_b} B_{2 \aeps \rho_j}(\eps z_j)  \Bigr) \leq \delta \eps \Bigr\}
  \end{aligned} \\
 &:= I_b^\eps \cup E^\eps \cup C^\eps .
\end{align*}
We already know $ \eps^d \# (I_b^\eps) \to 0$.
Next, we argue that 
\begin{align}\label{estimate.set.1}
\eps^d \#( E^\eps ) \rightarrow 0.
\end{align}
This follows by an argument similar to the one for \eqref{number.tildeI_b}: Then, we may choose $\eps_0=\eps_0(d)$ such that for all $\eps \leq \eps_0$, property \eqref{r.eps.correction} yields $\eta_\eps= \eps r_\eps \leq \delta \eps$. By definition of $J_b^\eps$ in \eqref{index.set.ppp.1} and of $E^\eps$ above, we infer that for such $\eps \leq \eps_0$, for all $z_j \in E^\eps$ there exists $z_i \in J_b^\eps$ such that
\begin{align}\label{inclusion.E}
	B_{\eta_\eps}(\eps z_j) \subset B_{2 \delta\eps + 2 \aeps \rho_i}(\eps z_i) \subset B_{6 \delta r_\eps^{-1} \aeps \rho_i}(\eps z_i),
\end{align}
where in the second inequality we use that $r_\eps^{-1}\delta \geq  1$.  We note that by \eqref{good.set.ppp} the balls $B_{\eta_\eps}(\eps z_j)$ with $z_j \in n^\eps$ are all disjoint. Hence,
\begin{align*}
\eps^d \#( E^\eps )= r_\eps^{-d} \eta_\eps^d \#( E^\eps )& \stackrel{\eqref{inclusion.E}}{\lesssim} r_\eps^{-d} \biggl | \bigcup_{z_i \in J^\eps_b} B_{6 \delta r_\eps^{-1} \aeps \rho_i}(\eps z_i) \biggr| \\
& \lesssim \delta^d r_\eps^{-2d} \sum_{z_i \in J^\eps_b} \Bigl(\aeps \rho_i \Bigr)^d 
\stackrel{\eqref{definition.r}}{\leq}\delta^d \eps^d \sum_{z_i \in J^\eps_b} \rho_i^{d-2} .
\end{align*}
By Lemma \ref{l.LLN.index.set}  and \eqref{number.J_b}, almost surely the right hand side tends to zero in the limit $\eps \downarrow 0^+$.

\medskip 
We conclude the argument for  \eqref{small.distance.bad} by showing that the set $C^\eps$ is empty when $\eps$ is small: In fact, by construction, if $z_i \in n_\eps$ satisfies
\begin{align}
	\dist\biggl( \eps z_i , \bigcup_{z_j \in \tilde I^\eps_b \cup K^\eps_b} B_{2 \aeps \rho_j}(\eps z_j)  \biggr) \leq \delta \eps,
\end{align}
then there exists a $z_j \in \tilde I^\eps_b \cup K^\eps_b$ such that for $\eps \leq \eps_0$
\begin{align*}
\eps |z_i - z_j| \leq \mbox{dist}\Bigl( \eps z_i, B_{2 \aeps \rho_j}(\eps z_j) \Bigr) + \eta_\eps \leq 2\delta \eps.
\end{align*}
 This yields $C^\eps \subset \Phi^\eps(D) \backslash \Phi_{2 \delta}(D)$ and thus that it is empty since by definition we also have $C^\eps \subset \Phi^\eps_{2\delta}(D)$. 
The proof of Lemma \ref{l.geometry.ppp} is complete.
\end{proof}

\bigskip

\begin{proof}[Proof of Lemma \ref{l.test.periodic}, general case. ] We split the proof of the lemma in the same steps as in the proof for the case of periodic centres (case (a)).
Some of these steps may be proven exactly as in the previous subsection by relying on Lemma \ref{l.geometry.ppp} instead of Lemma \ref{l.geometry.periodic}. We thus focus below only on the parts of the proof which differ from Subsection \ref{periodic}.

\medskip

\emph{Step 1.} Since by Lemma \ref{l.geometry.ppp}, the sets $H^\eps_g$ and $D^\eps_b$ are disjoint, the splitting $w^\eps= w^\eps_1 \wedge w^\eps_2$ with
$w^\eps_1, w^\eps_2$ solving \eqref{splitting.bad}, \eqref{splitting.good} and \eqref{reduction.bad} remains unchanged from the case of periodic centres.

 \medskip
 
\emph{Step 2.} Again by Lemma \ref{l.geometry.ppp}, we construct the sequence $\{w^\eps_1 \}_{\eps> 0}$ satisfying \eqref{splitting.bad} and \eqref{reduction.bad} as
in Subsection \ref{periodic}. We thus only need to focus on the construction of the functions $\{ w^\eps_2\}_{\eps>0}$, which we set equal to $1$ on $D^\eps_b$. 
For each $z_j \in n^\eps$, with $n^\eps$ being  the set of centers of the particles in $H^\eps_g$ (see Lemma \ref{l.geometry.ppp}), we denote
the random variables 
\begin{align}
	\label{def.d^eps_i}
d^\eps_j := \min\Bigl\{\operatorname{dist}( \eps z_j, D^\eps_b), \frac{1}{2}\min_{ i \neq j} \eps |z_i - z_j|, \eps \Bigr\} .
\end{align}
We remark that, in contrast with case (a) where we had by Lemma \ref{l.geometry.periodic} that $d^\eps_j \geq \frac{\eps}{2}$, here Lemma \ref{l.geometry.ppp} only implies that, for $\eps$ small, $d^\eps_j \geq r_\eps\eps$ with $r_\eps$ satisfying \eqref{r.eps.correction}.
By defining for each $z_j \in n^\eps$ the sets 
$$
T_j^\eps = B_{\eps^{\frac{d}{d-2}} \rho_j} (\eps z_j) \ \ \ \ B_j =  B_{d^\eps_j}(\eps z_j),
$$
we consider the function $w_2^{\eps,j}$ solving \eqref{cell.problem.0} in $B^j \backslash T^j$ and hence defined as
\begin{align}\label{ppp.cell.problem}
	w^{\varepsilon, j}_2(x) = \begin{cases}
		 \frac{ |x- \eps z_j|^{-(d-2)} - \left(d^\eps_j\right)^{-(d-2)}}{\eps^{-d} \rho_i^{-(d-2)} - \left(d^\eps_j\right)^{-(d-2)}} \quad \text{in} ~ B_j \setminus T_j \\
		 	 1 \quad \text{in} ~ T_j\\
	 0 \quad \text{in} ~ D \setminus B_j.
	\end{cases}
\end{align}

Note, that by definition of $d^\eps_j$, \eqref{def.d^eps_i}, the functions $\nabla w^{\eps,j}$ have disjoint support.
Moreover, for $\eps$ sufficiently small,
\begin{align}
	\label{big.small.radius}
	d^\eps_j \geq  2 \aeps \rho_j.
\end{align}
Indeed, by Lemma \ref{l.geometry.ppp},
\begin{align}
	2 \aeps \rho_j \leq \eps r_\eps \leq \min \Bigl\{\frac{1}{2}\min_{ i \neq j} \eps |z_i - z_j|, \eps\Bigr\},
\end{align}
and
\begin{align}
	2 \aeps \rho_j \leq  \aeps \rho_j + \frac {\eps r_\eps}{2} \leq \aeps \rho_j + \operatorname{dist}( T_j, D^\eps_b)
	= \operatorname{dist}( \eps z_j, D^\eps_b).
\end{align}

We thus set 
\begin{align}\label{w.2.ppp}
w^\eps_2 = 1 - \sum_{z_j \in  n^\eps} w^{\eps,j}_2,
\end{align}
which immediately satisfies condition \eqref{splitting.good}. Therefore, as discussed in Step 1, the function $w^\eps= w^\eps_1 \wedge w^\eps_2$ satisfies (H1) and it suffices
to prove (H2)-(H3) only for $w^\eps_2$.

\medskip

\emph{Step 3.} We begin by showing that $w^\eps_2$ satisfies (H2): By the triangular inequality 
and definitions \eqref{ppp.cell.problem} and \eqref{w.2.ppp},  we estimate 
\begin{align}\label{estimate.gradients}
\| \nabla w^\eps_2 \|^2_{2} = \sum_{z_i \in  n^\eps} \| \nabla w^{\eps,i}_2\|^2_{L^2(B_i)} 
&\lesssim \sum_{z_i \in n^\eps} \frac{\eps^d \rho_i^{d-2}}{1- \Big(\frac{\aeps \rho_i}{d_\eps}\Big)^{d-2} } \notag\\
&\stackrel{\eqref{big.small.radius}}{\lesssim}
 \sum_{i \in n^\eps}{\eps^d\rho_i^{d-2}}\stackrel{n^\eps \subset \Phi^\eps(D)}{\lesssim}\sum_{i \in \Phi^\eps (D)}{\eps^d\rho_i^{d-2}} .
\end{align}
By Lemma \ref{SLLN.general.pp}, the right-hand side above is almost surely bounded in the limit $\eps \downarrow 0^+$. This, together with Poincar\'e's inequality for
$1- w^\eps_2$ in $D$, yields that almost surely, up to a subsequence, we have $w^\eps_2 \rightharpoonup w$ in $H^1(D)$ when $\eps \downarrow 0^+$. 

\smallskip

We claim that $w\equiv 1$. To this purpose, it is useful to consider  the following ``truncated'' processes  $(n^\eps_M, \{\rho_{j,M} \}_{j\in n^\eps})$ which
we construct in the following way: For any $M \in \N$, we set
\begin{align}
n^\eps_{M}:= \Bigl\{ z_i \in  n^\eps \colon d_j^\eps \geq \frac{\eps}{M}\Bigr\}, \ \ \ \ \rho_{j,M} = \rho_j \wedge M.
\end{align}
In addition, let
\begin{align*}
H^{\eps,M}_g := \bigcup_{z_j \in n^{\eps}_M} B_{\aeps\rho_{j,M}}(\eps z_j), \ \ \ D^{\eps,M}:= D \backslash ( H^{\eps,M}_g \cup H^\eps_b),
\end{align*}
and let $w_2^{\eps,M}$ be the function constructed as in \eqref{w.2.ppp} and \eqref{ppp.cell.problem} for the set $H^{\eps,M}_g$.

By the same argument above for $w^\eps_2$, almost surely and up to a subsequence ,it holds $w_2^{\eps,M} \rightharpoonup w^M$ for every $M \in \N$. Moreover, since $1 - w^{\eps,M}_{2}=0$ on $\Rd \backslash \bigl( \bigcup_{z_i \in n^{\eps}_M} B_i \bigr)$ and the balls $B_i$ are pairwise disjoint and have radii in $[M^{-1}\eps, \eps]$ e may argue as in Subsection \ref{periodic} and infer that almost surely, and for every $M\in \N$, $w^{\eps,M}_2 \wto 1$ in $H^1(D)$, and therefore also strongly in $L^2(D)$.  This implies by the triangular inequality 
\begin{align*}
\limsup_{\eps \downarrow 0^+} \| w^\eps_2 - 1 \|_{2}^2 \leq \limsup_{M \uparrow \infty} \limsup_{\eps \downarrow 0^+} \bigl\| w^\eps_2 - w_2^{\eps,M} \bigr\|_2^2 .
\end{align*}
Condition (H2) holds for $w^\eps_2$ provided that the limit on the right-hand side above vanishes. By Poincar\'e's inequality in $D$, it suffices to prove that 
\begin{align}\label{calculation.gradients.ppp}
\lim_{M\uparrow \infty} \limsup_{\eps \downarrow 0^+} \| \nabla ( w^\eps_2 - w_2^{\eps,M}) \|_2^2 = 0.
\end{align}
To show this, we argue as follows: By construction
\begin{align}
&w^\eps_2 = w_2^{\eps,M} \ \ \ \text{ in $B_i \subset D\backslash D^\eps_b$, whenever $\rho_i \leq M$ and $d_i \geq M^{-1}\eps$,}\\
&w_2^{\eps,M} \equiv 1 \ \ \ \ \ \ \text{in $B_i$, whenever $d_i \leq M^{-1}\eps$.}
\end{align}
This implies that the $L^2$-norm on the left-hand side above reduces to
\begin{align}\label{measure.conv.truncated}
 \bigl\| \nabla \bigl( w^\eps - w_2^{\eps,M}\bigr) \bigr\|_2^2 &= \sum_{z_i \in n^\eps}  \bigl\| \nabla \bigl( w_2^{\eps,i} - w_2^{\eps,M,i}\bigr) \bigr\|_2^2 \1_{\rho_i \geq M} \1_{ d_i \geq M^{-1}\eps }
\notag\\&\quad \quad
  +   \sum_{z_i \in n^\eps} \| \nabla w_2^\eps \|_2^2 \1_{d_i \leq M^{-1}\eps }.
\end{align}
Similarly to \eqref{estimate.gradients}, we use the explicit formulation for $w^\eps_2$, $w^{\eps,M}_2$ to control the first term on the right-hand side above by
\begin{align*}
\sum_{z_i \in n^\eps}& \bigl\| \nabla \bigl( w_2^{\eps,i} - w_2^{\eps,M,i}\bigr) \bigr\|_2^2 \1_{\rho_i \geq M} \1_{ d_i \geq M^{-1}\eps }
{\lesssim}  \sum_{z_i \in n^\eps}\eps^d \rho_i^{d-2} \1_{\rho_i \geq M}.
\end{align*}
By Lemma \ref{SLLN.general.pp} applied to the process $\Phi$ with marks $\{ \rho_i^{d-2} \1_{\rho_i \geq M} \}_{z_i \in \Phi}$, and the assumption \eqref{AverRadii}, we
obtain that almost surely
\begin{align}\label{measure.conv.truncated.1}
\lim_{M\uparrow +\infty}\limsup_{\eps\downarrow 0^+}\sum_{z_i \in n^\eps}& \bigl\| \nabla \bigl( w_2^{\eps,i} - w_2^{\eps,M,i}\bigr) \bigr\|_2^2 \1_{\rho_i \geq M} \1_{ d_i \geq M^{-1}\eps }
=  0.
\end{align}
By using again the same estimate as in \eqref{estimate.gradients} for $w^\eps_2$, we have that
\begin{align}
	\label{small.distance.gradients}
\sum_{z_i \in n^\eps}\| \nabla w_2^{\eps,i} \|_2^2 \1_{d_i \leq M^{-1}\eps } 
&{\lesssim}  \sum_{z_i \in n^\eps} \eps^d \rho_i^{d-2} \1_{d_i \leq M^{-1}\eps}.
\end{align}
By Definition of  $d_i$ in \eqref{def.d^eps_i}, we have that if $d_i \leq M^{-1} \eps$, then either 
$z_i \in \Phi^\eps(D) \backslash \Phi_{2 M^{-1}}^\eps(D)$ or 
\begin{align}
z_i \in  I^\eps_M := \Bigl\{  z_i \in n_\eps \cap \Phi^\eps_{2 M^{-1}}(D) \, \colon \,\dist( z_i , D^\eps_b ) \leq \frac \eps M \Bigr\}.
\end{align} 

Hence, from  \eqref{small.distance.gradients} we obtain
\begin{align}
\limsup_{\eps \downarrow 0+} \sum_{z_i \in n^\eps}\| \nabla w_2^{\eps,i} \|_2^2 \1_{d_i \leq M^{-1}\eps } 	
&\leq \limsup_{\eps \downarrow 0+} \eps^{d} \sum_{ z_i \in \Phi^\eps(D) \backslash \Phi^{2 M^{-1}, \eps}(D)} \rho_i^{d-2} +  \limsup_{\eps \downarrow 0^+}\eps^{d} \sum_{ z_i \in I^\eps_M} \rho_i^{d-2} \\
\end{align}
On the one hand, by \eqref{small.distance.bad} and Lemma \ref{l.LLN.index.set}, the second term on the right-hand side vanishes. On the other hand, Lemma \ref{SLLN.general.pp} for both $\Phi$ and
 $\Phi^{2M^{-1}}$ imply that
\begin{align}
	\limsup_{\eps \downarrow 0+} \sum_{z_i \in n^\eps}\| \nabla w_2^{\eps,i} \|_2^2 \1_{d_i \leq M^{-1}\eps } 
	\leq \langle  N(Q) - N_{2 M^{-1}}(Q) \rangle \langle \rho^{d-2} \rangle,
\end{align}
where $Q$ is a unit cube. Finally, the right-hand side in the above estimate converges to zero in the limit $M \to \infty$ again by  Lemma \ref{SLLN.general.pp}.
If we now wrap up the previous estimate with \eqref{measure.conv.truncated.1} and \eqref{measure.conv.truncated}, we conclude  \eqref{calculation.gradients.ppp}. 
We thus established (H1) and (H2) for the sequence $w^\eps_2$ (and thus also fo $w^\eps$).

\medskip

It remains to prove (H3). 
 We consider again the truncated sequences $\{w^{\eps,M}_2\}_{\eps>0}$ above and start by arguing that it is enough to show that, for every $M \in \N$ fixed, condition (H3) is
 satisfied by $w^{\eps,M}_2 $, namely
\begin{align}\label{H3.truncated}
	\bigl( -\Delta w^{\eps,M}_2, v_\varepsilon \bigr)_{H^{-1}, H^{1}_0} \rightarrow C_{0,M}\int_D v,
\end{align}
where $C_{0,M}:= (d-2) \sigma_d \langle N_{2M^{-1}}(Q) \rangle \langle \, \rho_M^{d-2} \, \rangle$. In fact, Cauchy-Schwarz's
inequality yields 
\begin{align*}
 \bigl|\bigl( -\Delta(w^\eps_2 -  w^{\eps,M}_2), v_\varepsilon \bigr)_{H^{-1}, H^{1}_0}\bigr| 
& \leq  \biggl( \int |\nabla v_\eps|^2\biggr)^{\frac 1 2}\biggl( \int |\nabla ( w^\eps_2 - w^{\eps,M}_2)|^2\biggr)^{\frac 1 2},
\end{align*}
and, as the family $\{ v^\eps\}_{\eps>0}$ is uniformly bounded in $H^1(D)$, by \eqref{calculation.gradients.ppp} we get
\begin{align*}
\lim_{M\to \infty} \limsup_{\eps \downarrow 0^+} |( -\Delta(w^\eps_2 -  w^{\eps,M}_2), v_\varepsilon )_{H^{-1}, H^{1}_0}| =0.
\end{align*}
Since $C_{0,M} \to C_0$ when $M \uparrow +\infty$ by \eqref{convergence.average} of Lemma \ref{SLLN.general.pp}, the above limit and the triangular inequality yield that to prove (H3) it suffices to show  \eqref{H3.truncated}.

\medskip

 The proof of \eqref{H3.truncated} follows the same lines of  Step 3. in Subsection \ref{periodic} (in particular, \eqref{convergence.mu} for $w^{\eps,M}_2$), so we just point out the differences: By arguing as in that case, it suffices to prove that
\begin{align}
	\label{H3.eta}
	\eta_M^\eps := \sum_{z_i \in n^\eps_M}	d(d-2)\rho_{i,M}^{d-2} \frac{\eps^d}{d_i^d} \1_{B_i} 
	&\stackrel{*}{\rightharpoonup} C_{0,M} \quad \text{in $L^\infty(D)$}
\end{align}
The factor $\eps^d d_i^{-d}$ in the above expression is due to the fact that the balls $B_i$ have now radii $d_i$ instead of $\eps$.
Hence, by including this factor, we have 
\begin{align}
	\label{L^1.norm.indicator.ball}
	\|\eps^d d_i^{-d} \1_{B_i}\|_{L^1} = \|\1_{B_\eps(\eps z_i)}\|_{L^1} = \eps^d
\end{align}
as in the periodic case.

Since
\begin{align}
	\| \eta_M^\eps\|_{L^\infty} \lesssim  M^{d(d-2)}, 
\end{align}
to show \eqref{H3.eta} it suffices to test $\eta_M^\eps$ with functions $\zeta \in C^1_0(D)$.
We observe that if we define
\begin{align}
	\tilde\eta_M^\eps := \sum_{z_i \in \Phi^\eps_{2M^{-1}}  (D)} d(d-2)\rho_{i,M}^{d-2} \frac{\eps^d}{d_i^d} \1_{B_i} ,
\end{align}
then as in Step 3 of Subsection \ref{periodic}, we use Lemma \ref{SLLN.general.pp} for $\Phi_{2M^{-1}}$ and Lemma \ref{l.w.independent} applied to $(\Phi_{2M^{-1}}, \{\rho_{i,M}^{d-2}\})$
to infer that, almost surely and for all $\zeta \in C_0^1(D)$,
\begin{align}
	\label{H3.etatilde}
	\int_D  \tilde\eta_M^\eps \zeta \to C_{0,M} \int_D \zeta.
\end{align}
We conclude \eqref{H3.eta} by arguing that, almost surely and for all
$\zeta \in C^1_0(D)$, we have
\begin{align}
	\left|\int_D (\eta_M^\eps - \tilde\eta_M^\eps) \zeta\right| \to 0.
\end{align}
Indeed,
\begin{align}
\left|\int_D (\eta_M^\eps - \tilde\eta_M^\eps) \zeta\right| 
&\lesssim  M^{d-2} \sum_{z_i \in \Phi^\eps_{2M^{-1}}(D) \backslash  n^\eps_M} \int_{B_\eps(\eps z_i)}  |\zeta| \\
& \lesssim M^{d-2} \| \zeta\|_\infty \eps^d  \# \Bigl( \Bigl\{ z_i \in \Phi^\eps_{2M^{-1}} (D) \colon d_i \leq  \frac \eps M\Bigr\} \Bigr).
\end{align} 
By definition of $ d_i$ and \eqref{small.distance.bad}, the right-hand side tends to zero almost surely in the limit $\eps \to 0$.
This concludes the proof of \eqref{H3.eta} and thus establishes (H3) for the sequence $\{ w^{\eps}_2 \}_{\eps>0}$. 
The proof of Lemma \ref{l.test.periodic} is complete.
\end{proof}
\section{Auxiliary results}\label{aux}
In this section, we prove some auxiliary results on which  the proofs of Section \ref{section.oscillating} rely. 
We give these results for a general marked point process $(\Phi , \X)$ with $\Phi$ satisfying \eqref{finite.variance} and \eqref{mixing} and with the marks
$\X:=\{X_i \}_{z_i\in \Phi}$ satisfying \eqref{correlation.radii.0} with 
\begin{align}\label{correlation.radii.1}
\langle X \rangle = \int_{0}^{+\infty} x h(x) dx < +\infty
\end{align}
and with the function $g$ being bounded as in \eqref{correlation.radii.2} (with $\rho$ substituted by $x$ and with $p > 2$).  

 \begin{lem}\label{SLLN.general.pp}
Let $Q$ a unitary cube and let $(\Phi, \X)$ be a marked point process as introduced above.
Then, for every bounded set $B \subset \Rd$ {which is star-shaped with respect to the origin}, we have 
 \begin{align}\label{average.number}
 \lim_{\eps \downarrow 0^+} \eps^d N^\eps(B)= \langle N(Q) \rangle  |B| \ \ \ \ \text{almost surely},
 \end{align}
and
\begin{align}\label{SLLN.convergence.pp}
\lim_{\eps \downarrow 0^+} \eps^{d}\sum_{ z_i \in \Phi^\eps(B)} X_i =  \langle N(Q) \rangle \langle X \rangle |B| \ \ \ \ \text{almost surely.}
\end{align}
 
 \smallskip
 
Furthermore, for every $\delta < 0$ the process $\Phi_\delta$ obtained from $\Phi$ as in \eqref{thinned.process} satisfies the analogues of \eqref{SLLN.convergence.pp}, \eqref{average.number} and
\begin{align}\label{convergence.average}
\lim_{\delta \downarrow 0^+}\langle N_\delta(A) \rangle = \langle N(A) \rangle
\end{align}
for every bounded set $A \subset \Rd$.

\end{lem}

  \begin{lem}\label{l.LLN.index.set}
	In the same setting of Lemma \ref{SLLN.general.pp}, let $\{ I_\eps \}_{\eps > 0}$ be a family of collections of points such that $I_\eps \subset \Phi^\eps( B)$ and
       \begin{align}\label{small.index.set}
       \lim_{\eps \downarrow 0^+} \eps^d\# I_\eps = 0 \ \ \ \text{almost surely.}
       \end{align}
       Then,
	\begin{align}
		\lim_{\eps \downarrow 0^+}{ \eps^d} \sum_{z_i \in I_\eps} X_i \to 0 \quad \text{almost surely.}
	\end{align}
\end{lem}

\bigskip
  \begin{lem}\label{l.w.independent}
In the same setting of Lemma \ref{SLLN.general.pp}, let us assume that in addition the marks satisfy $\langle X^2 \rangle < +\infty$. 
For $z_i \in \Phi$ and $\eps >0$, let $r_{i,\eps} > 0$, and assume there exists a constant $C>0$  such that for all $z_i \in \Phi$ and $\eps >0$
\begin{align}
	r_{i,\eps} \leq C \eps.
\end{align}
Then, almost surely, we have
\begin{align*}
\lim_{\eps \downarrow 0^+ }\sum_{z_i \in \Phi^\eps(B)} X_i \frac{\eps^d}{r_{i,\eps}^d}\int_{B_{r_{i,\eps}}(\eps z_i)} \zeta(x) \, dx = \frac{\sigma_d}{d} \langle N(Q) \rangle \langle \, X \, \rangle \int_{B} \zeta(x) \, dx,
\end{align*}
for every $\zeta \in C_0^{1}(B)$. 
 \end{lem}
 
 \bigskip

 \begin{proof}[Proof of Lemma \ref{SLLN.general.pp}]
 Without loss of generality we assume that $B= Q^R$, with $Q^R$ the cube of size $R$ centred at the origin and ${Q^{ \frac R \eps}=\frac 1 \eps}B$. Moreover, we denote by $\{ Q_i \}_{i\in \Zd}$ the partition of $\Rd$ made of (essentially) disjoint unit cubes centred in the points of the lattice $\Zd= \{x_i \}_{i\in \N}$.
 
 \medskip
 
The limit \eqref{SLLN.convergence.pp} is an easy consequence of a Strong Law of Large Numbers for correlated random variables (see Lemma \ref{SLLN.correlated} in the Appendix): For all $\mu > 0$ and all $\eps$ small enough
\begin{align}\label{limsup}
\eps^{d}\sum_{ z_i \in \Phi^\eps( Q^{R})} X_i \leq \eps^{d}\sum_{x_j \in \Zd \cap Q^{\frac {R+\mu} \eps}} Z_i,
\end{align}
where  $Z_j:= \sum_{ z_i  \in \Phi( Q_j)} X_i$ are identically distributed random variables by stationarity of $(\Phi, \X)$.  Moreover, they have finite average 
\begin{align}\label{average.Z}
\Bigl\langle  \sum_{ z_i  \in \Phi( Q)} X_i \Bigr\rangle = \langle N(Q) \rangle \langle X \rangle < +\infty
\end{align}
and satisfy for every $i, j \in \N$ with $i \neq j$
\begin{align}\label{correlation.sum.1}
|\langle Z_i Z_j \rangle - \langle Z \rangle^2| &\stackrel{\eqref{average.Z}}{=} \biggl|\biggl\langle \sum_{ z_k  \in \Phi( Q_i) \atop z_l \in \Phi(Q_j)} X_k X_l  \biggr\rangle  - \langle N(Q) \rangle^2 \langle X \rangle^2\biggr| \\&\stackrel{\eqref{correlation.radii.2}}{\leq}| \langle X \rangle^2 \langle  N(Q_i) N(Q_j) \rangle -  \langle N(Q) \rangle^2 \langle X \rangle^2| + \frac{C}{|x_i - x_j|^\gamma}\langle  N(Q_i) N(Q_j) \rangle,
\end{align}
where the constant $C$ depends on the constants in \eqref{correlation.radii.2}.
We now appeal to condition \eqref{mixing}: By the stationarity assumption on $\Phi$, we have that for any $i,j \in \N$
\begin{align}
\langle N(Q_i ) N(Q_j) \rangle = \langle N( Q_{i-j}) N (Q) \rangle,
\end{align}
so that \eqref{mixing} applied to the random variables $N( Q_{i-j})$ and $N (Q)$ yields
\begin{align}\label{mixing.numbers}
|\langle N(Q_i) N( Q_j) \rangle - \langle N(Q) \rangle^2 | \lesssim \frac{\langle N(Q )^2 \rangle}{| x_i - x_j|^\gamma}.
\end{align}
We thus insert this bound into \eqref{correlation.sum.1} and get
\begin{align}
|\langle Z_i Z_j \rangle - \langle Z \rangle^2| {\leq} \frac{C \langle N(Q )^2 \rangle }{|x_i - x_j|^\gamma}.
\end{align} 
Hence, condition \eqref{SLLN.mixing.standard} is satisfied with constant $C \langle N(Q )^2 \rangle  < +\infty$, where $C$ depends on the constants in \eqref{mixing} and \eqref{correlation.radii.0}. 
We apply Lemma \ref{SLLN.correlated} to the sequence $\{Z_i \}_{i\in \Zd}$ in  \eqref{limsup} and conclude 
\begin{align}
	\limsup_{\eps \downarrow 0^+} \eps^{d}\sum_{ z_i \in \Phi^\eps( Q^ R )} X_i \leq \langle N(Q) \rangle \langle X \rangle |Q^{R+\mu}|.
\end{align}
Arguing analogously for the limit inferior and taking the limit $\mu \to 0$ yields \eqref{SLLN.convergence.pp}. 
Limit \eqref{average.number} follows exactly as \eqref{SLLN.convergence.pp} by substituting the marks  $X_i$ with $1$.

\medskip

 For $\delta > 0$ be fixed. We show that $\Phi_\delta$ satisfies the analogues of \eqref{average.number} and \eqref{SLLN.convergence.pp} together with \eqref{convergence.average}. Since by definition \eqref{thinned.process} we have that $N^\delta(B) \leq N(B)$, the limit in \eqref{convergence.average}  follows from the Dominated Convergence Theorem. 
To show \eqref{SLLN.convergence.pp} and \eqref{average.number} we may argue exactly as above for the original process $\Phi$ and apply Lemma \ref{SLLN.correlated} to the random variables $Z^\delta_i := \sum_{ z_i  \in \Phi^\delta( Q_j)} X_i$. Since for each $x_i\in \Zd$ we have $0\leq Z^\delta_i \leq Z_i$, the only condition that remains to be shown for  the collection $\{ Z^\delta_i\}_{x_i\in \Zd}$ is \eqref{SLLN.mixing.standard}. By arguing as in \eqref{correlation.sum.1}, we use again \eqref{correlation.radii.2} to reduce ourselves to show \eqref{SLLN.mixing.standard} for the random variables $\{N^\delta(Q_i)\}_{i \in \Zd}$. To do so, for any $x\in \Rd$ we define
	\begin{align}\label{minimal.distance.def}
	d_x:=  \min_{y \in \Phi(\omega), \atop y \neq x} | x- y|,
	\end{align}
	so that
	$$
	N^\delta( Q) = \sum_{z_i \in \Phi \cap Q} \1_{d_x > \delta}(z_i), \ \ \ N^\delta(Q_i)=  \sum_{z_i \in \tau_{-x_i}\Phi \cap Q} \1_{d_x > \delta}(z_i).
	$$
	Since $\1_{d_x  > \delta} = \1_{N( B_\delta(x) \backslash \{x \} = 0)}$, each $N^\delta(Q_i)$ are measurable random variables with respect to $\mathcal{F}( B_\delta(Q_i))$ defined in \eqref{mixing}, with 
	$$
	B_\delta(Q_i):= \{ x \in \Rd \ \colon \ \operatorname{dist}(x, Q_i) \leq \delta \}.
	$$ 
We thus apply \eqref{mixing} as above and conclude that, with a constant depending on $\delta$, condition  \eqref{SLLN.mixing.standard} is satisfied by the sequence $N^\delta(Q_i)$. This yields \eqref{SLLN.convergence} for $\Phi^\delta$ and, by the same argument, also \eqref{average.number}. 
 \end{proof}

 \begin{proof}[Proof of Lemma \ref{l.LLN.index.set}]
	Let $M \in \N$. We define for every $z_i\in \Phi$ the truncated marks \mbox{$Y_i := X_i\1_{[M,\infty)}(X_i)$} which satisfy assumption \eqref{correlation.radii.0} and \eqref{correlation.radii.2} thanks to the corresponding assumptions for the original marks $\{X_i\}_{i\in \N}$. Since
	$$
	\langle Y_i \rangle \leq \langle X \rangle < +\infty,
	$$
	 we apply Lemma \ref{SLLN.general.pp} to the point process $\Phi$ with the truncated marks $\{ Y_i\}_{z_i\in \Phi}$ to infer that almost surely
	 $$
	 \eps^d \sum_{z_i \in \Phi^\eps(B)} Y_i \to \langle X \1_{[M,\infty)}(X) \rangle.
	 $$
	 This yields
	\begin{align}
		\limsup_{\eps \downarrow 0^+} \eps^d \sum_{z_i \in I_\eps} X_i 
		&\leq \limsup_{\eps \downarrow 0^+}  \eps^d \sum_{z_i \in I_\eps} X_i \1_{[0,M)}(X_i) +  \langle X \1_{[M,\infty)}(X) \rangle \\
		& \leq  M\limsup_{\eps \downarrow 0^+} \eps^d \# I_\eps + \langle X \1_{[M,\infty)}(X) \rangle \stackrel{\eqref{small.index.set}}{=} \langle X \1_{[M,\infty)}(X) \rangle.
	\end{align}
	Since $\langle X \rangle < +\infty$, we may take the limit $M \to \infty$ and conclude the proof.
\end{proof}

\begin{proof}[Proof of Lemma \ref{l.w.independent}]
First, we argue that it suffices to prove the case $r_{i,\eps} = \eps$ for all $z_i \in \Phi$ and $\eps >0$.
Indeed, for $\zeta \in C^{1}_0(B)$ we use a change of coordinates to get almost surely,
\begin{align}
	\limsup_{\eps \downarrow 0^+} \sum_{z_i \in \Phi^\eps (B)} \biggl|\frac{\eps^d}{r_{i,\eps}^d}\int_{B_{r_{i,\eps}}(\eps z_i)} \zeta(x) \dd x - \int_{B_{\eps}(\eps z_i)} \zeta(x) \dd x\biggr| 
	\leq \limsup_{\eps \downarrow 0^+} C \eps \| \nabla \zeta\|_{L^\infty} \eps^d N^\eps(B) = 0
\end{align}
since $\eps^d N^\eps(B)$ is bounded by Lemma \ref{SLLN.general.pp}.
 
\medskip

Without loss of generality we therefore assume $r_{i,\eps} = \eps$ and $|B| = 1$. 

\medskip

Next we observe that it suffices to argue that the assertion holds for any fixed $\zeta \in W^{1,\infty}_0(B)$. Indeed, once we have shown this,
the statement follows because there exists a countable subset of $W^{1,\infty}_0(B)$ which is dense in $C^1_0(B)$.

\medskip

We fix $\zeta \in  W^{1,\infty}_0(B)$ and we begin by rewriting the term in the limit as 
\begin{align*}
\sum_{z_i \in \Phi^\eps (B)} X_i \int_{B_\eps(\eps z_i )} \zeta( x) \, dx& = \sum_{z_i \in \Phi^\eps (B)} (X_i - \langle X \rangle )\int_{B_\eps(\eps z_i )} \zeta( x) \, dx  \\
&\quad +  \langle X \rangle  \sum_{z_i \in \Phi^\eps (B)} \int_{B_\eps(\eps z_i )} \zeta( x) \, dx,
\end{align*} 
so that
\begin{align}\label{riemann.sum.2}
\biggl|&\sum_{z_i \in \Phi^\eps (B)} X_i \int_{B_\eps(\eps z_i )} \zeta( x) \, dx - \frac{\sigma_d}{d} \langle N(Q) \rangle \langle X\rangle \int_B \zeta \biggr|\\
&\quad \quad \leq \biggl|\sum_{z_i \in \Phi^\eps (B)} (X_i - \langle X \rangle )\int_{B_\eps(\eps z_i )} \zeta( x) \, dx\biggr|  +\langle X \rangle \biggl| \sum_{z_i \in \Phi^\eps (B)} \int_{B_\eps(\eps z_i )} \zeta( x) \, dx - \langle N(Q) \rangle \int_B \zeta \biggr|.\notag
\end{align} 
Let $\{ Q_i\}_{i\in \N}$ be a partition of $\Rd$ into (essentially) disjoint unitary cubes and let $\{ y_i \}_{i\in \N}$ be the collection of their centres. We claim that if
\begin{align*}
   T_\eps(\zeta) &:= \int_B \zeta ,   &&\tilde{T}_\eps(\zeta) := \eps^d \sum_{Q_i \cap \frac 1 \eps B \neq \emptyset} \zeta(\eps y_i),\\
   R_\eps(\zeta) &:= \sum_{z\in \Phi^\eps( B)} \int_{B_\eps( \eps z)} \zeta(x) \, dx,  &&\tilde{R}_\eps(\zeta) := \eps^d \frac{\sigma_d}{d} \sum_{Q_i \cap \frac 1 \eps B \neq \emptyset} N(Q_i)\zeta(\eps y_i),
\end{align*}
then
\begin{align}\label{riemann.sum.3}
&\lim_{\eps \downarrow 0^+} | T^\eps(\zeta) - \tilde T^\eps(\zeta) |= 0, \ \ \ \lim_{\eps \downarrow 0^+} | R^\eps(\zeta) - \tilde R^\eps(\zeta) | = 0 \ \ \ \ \ \text{almost surely.}
\end{align}
The first limit is a standard Riemann sum; for the second limit we argue in a similar way: Since $\zeta \in  W^{1,\infty}_0(B)$, we have that 
 \begin{align}
 		|R_\eps(\zeta) - \tilde{R}_\eps(\zeta)| &= \biggl|\sum_{Q_i \cap \frac 1 \eps B \neq \emptyset}\biggl( \sum_{z_j \in \Phi(Q_i)} \int_{B_\eps(\eps z_j)} \zeta - \eps^d \frac{\sigma_d}{d} N(Q) \zeta(\eps y_i) \biggr) \biggr|\\
		& = \biggl|\sum_{Q_i \cap \frac 1 \eps B \neq \emptyset}  \sum_{z_j \in \Phi(Q_i)} \int_{B_\eps(\eps z_j)} (\zeta(x) -  \zeta(\eps y_i) ) \biggr|\leq
		 2\| \nabla \zeta \|_{\infty} \, \eps^{d+1}  N^\eps(B).	
		 \end{align}
We now apply \eqref{average.number} of Lemma \ref{SLLN.general.pp} to infer \eqref{riemann.sum.3}. This, together with \eqref{riemann.sum.2} and the triangular inequality implies that almost surely
\begin{align}\label{riemann.sum.4}
&\limsup_{\eps \downarrow 0^+}\biggl|\sum_{z_i \in \Phi^\eps (B)} X_i \int_{B_\eps(\eps z_i )} \zeta( x) \, dx -\frac{\sigma_d}{d} \langle N(Q) \rangle \langle X\rangle \int_B \zeta \biggr|\notag \\
&\quad \quad \leq \limsup_{\eps \downarrow 0^+}\biggl|\sum_{z_i \in \Phi^\eps (B)} (X_i - \langle X \rangle )\int_{B_\eps(\eps z_i )} \zeta( x) \, dx\biggr|\\
&\quad \quad  + \limsup_{\eps \downarrow 0^+} \biggl| \eps^d \langle X \rangle \frac{\sigma_d}{d} \,  \sum_{Q_i \cap \frac 1 \eps B \neq \emptyset} \zeta(\eps y_i)  \bigl(N(Q_i)- \langle N(Q) \rangle \bigr) \biggr|.\notag
\end{align} 
It remains to show that also the previous two terms on the right-hand side above vanish almost surely. Since to do this we follow an argument very similar to the one of Lemma \ref{SLLN.correlated}, we only give the details of the parts in which the proof differs.  For $\eps >0$ let
$$
a_{i,\eps}:=  \int_{B_\eps(\eps z_i)} \zeta(x) \, dx , \ \ \ \ \ \tilde X_i := X_i - \langle X_i \rangle
$$
and
$$
S_\eps:= \sum_{z_i \in \Phi^\eps(B)} a_{i,\eps} X_i, \ \ \ \ \tilde S_\eps:= \sum_{z_i \in \Phi^\eps(B)} a_{i,\eps} \tilde X_i.
$$
 We start by proving that the the first term on the right-hand side of \eqref{riemann.sum.4}, i.e. $\tilde S_\eps$ above, vanishes in the limit; we may argue analogously that also second term on the right-hand side of \eqref{riemann.sum.4}  vanishes.

As in the proof of Lemma \ref{SLLN.correlated}, we may use Chebyshev's inequality to estimate for each $\delta >0$ 
 \begin{align}\label{Cheb}
 \P( \tilde S_{\eps} > \delta ) \leq \delta^{-2} \langle \tilde S_{\eps}^2 \rangle
 \end{align}
and rewrite 
 \begin{align*}
  \langle \tilde S_{\eps}^2 \rangle &= \biggl\langle \sum_{z_i, z_k \in \Phi( \frac 1 \eps B)} a_{i,\eps} a_{k,\eps} \tilde X_i \tilde X_k \biggr\rangle = \sum_{Q_j \cap \frac 1 \eps B \neq \emptyset \atop Q_i \cap \frac 1 \eps B \neq \emptyset}  \biggl\langle \biggr(\sum_{z_l \in \Phi( Q_j)} a_{j,\eps}  \tilde X_j  \biggr) \biggl( \sum_{z_k \in \Phi( Q_i )}  a_{k,\eps}\tilde X_k \biggr) \biggr\rangle.
  \end{align*}
If we now set $Y_i := \sum_{z_l \in \Phi( Q_j)} a_{j,\eps}  \tilde X_j $, since all $|a_{\eps,i}|\leq \|\zeta\|_{L^\infty}\eps^d$, we argue as for the random variables $\{ Z_i\}_{i\in \N}$ in the proof of Lemma \ref{SLLN.general.pp} and infer that 
 \begin{align*}
  \langle \tilde S_{\eps}^2 \rangle  &\leq \sum_{Q_i \cap \frac 1 \eps B \neq \emptyset}   \| \zeta \|_{\infty}^2 \eps^{2d} \langle N(Q)^2 \rangle   \mbox{Var}(X) +  C \eps^{2d} \|\zeta \|_{\infty}^2 \sum_{Q_j \cap \frac 1 \eps B \neq \emptyset \atop Q_i \cap Q_j = \emptyset } \frac{\langle N(Q)^2 \rangle \langle X \rangle^2}{|x_i - x_j|^{\gamma}}\\
  & \stackrel{\gamma > d}{\lesssim} \eps^{d} \|\zeta \|_{\infty}^2 \langle N(Q)^2 \rangle \langle X^2 \rangle.
  \end{align*}
Therefore, if we plug this into \eqref{Cheb} and apply Borel-Cantelli's lemma to the subsequence $\eps_n = \frac{1}{n}$ with $n\in \N$, we get that
\begin{align*}
\lim_{n \uparrow +\infty}  \tilde S_{\eps_n} =0 \ \ \ \text{almost surely.}
\end{align*}
We appeal to an estimate similar to this one also for the second term on the right hand side of \eqref{riemann.sum.4}  (this time using the assumption \eqref{mixing}) and conclude from \eqref{riemann.sum.4} that for the sequence $\{ \eps_n \}_{n\in \N}$ we have almost surely that
\begin{align}\label{limit.1}
\lim_{n \uparrow +\infty} \sum_{z_i \in \Phi (\frac 1 {\eps_n} B)} X_i \int_{B_{\eps_n}(\eps_n z_i  )} \zeta( x) \, dx =  \ \frac{\sigma_d}{d} \langle N( Q) \rangle \langle  X \rangle \int_{B} \zeta(x) \, dx. 
\end{align}

 \medskip
  
 To extend \eqref{limit.1} to any sequence $\eps_j \downarrow 0$ and for the same full-probability set, we argue again similarly to Lemma \ref{SLLN.correlated}. We first fix the following notation: 
 For $0 < \eps < 1$, we define 
 $$
 \underline \eps: = \Bigl( \lfloor \frac{1}{\eps} \rfloor + 1 \Bigr)^{-1} , \ \ \ \ \ \overline \eps :=  \Bigl( \lfloor \frac{1}{\eps} \rfloor \Bigr)^{-1}.
 $$
 Note that $\overline \eps^{-1}, \underline \eps^{-1} \in \N$ and $\underline \eps \leq \eps \leq \overline \eps $.
 By writing $\zeta = \zeta_+ + \zeta_-$ and using linearity, we observe that it suffices to consider non-negative functions $\zeta$ and thus reduce ourselves to the case $a_{i,\eps} \geq 0$.
 
 \medskip
 
For any $\eps_j \downarrow 0^+$ we may use the triangle inequality and the assumptions on the sign of the weights and the $X_i$'s to bound 
\begin{align}\label{l.w.1}
\ \ \ S_{\eps_j} \leq   S_{\underline\eps_j} + \sum_{i= 1}^{N^{\eps_j}(B)} |a_{i ,\eps_j} - a_{i, \underline \eps_j} | X_i  
\leq  S_{\underline \eps_j} + \max_{i=1, \cdots, N^{\eps_j}(B) } |a_{i ,\eps_j} - a_{i, \underline\eps_j} | \, \sum_{i= 1}^{N^{\underline\eps_j}(B)}X_i .
\end{align}
We now claim that the weights are uniformly continuous in the second index, uniformly in the first index: More precisely we have that almost surely
 \begin{align}\label{UC.weights}
\lim_{\eps\downarrow 0^+} \frac{\max_{i \leq \#N^{\overline\eps}(B)} |a_{i, \eps} - a_{i, \overline \eps}|}{\overline\eps^d} = \lim_{\eps\downarrow 0^+} \frac{\max_{i \leq \#N^{\eps}(B)} |a_{i, \eps} - a_{i, \underline \eps}|}{\underline\eps^d}=0.
 \end{align}
We first argue that, if this is true, the proof of the lemma is concluded:  From \eqref{l.w.1} we indeed obtain
\begin{align*}
\ S_{\eps_j}  \leq S_{\underline \eps_j}  + \frac{\max_{i=1, \cdots, N^{\eps_j}(B) } |a_{i ,\eps_j} - a_{i, \underline \eps_j} |}{\underline\eps_j^d} \, \underline \eps_j^d \sum_{i= 1}^{N^{\underline\eps_j}(B)}X_i.
\end{align*}
Limit \eqref{limit.1}, Lemma \ref{SLLN.general.pp} and the second limit in \eqref{UC.weights} yield
\begin{align*}
\limsup_{\eps_j \downarrow 0} S_{\eps_j} \leq \frac{\sigma_d}{d} \langle N( Q) \rangle \,  \langle X \rangle\, \int \zeta.
\end{align*}
We may argue similarly as in \eqref{l.w.1} for the bound from below and get that
\begin{align*}
\liminf_{\eps_j \downarrow 0} S_{\eps_j} \geq \frac{\sigma_d}{d}  \langle N(Q ) \rangle  \, \langle X \rangle \int \zeta .
\end{align*}
This yields the claim of Lemma \ref{l.w.independent}.

\medskip

It thus remains to establish \eqref{UC.weights}: Since $\zeta  W^{1,\infty}_0(B)$, for any choice of $z_i \in B$ and $\eps_1 \leq \eps_2$ we estimate
\begin{align*}
|a_{i, \eps_1} - a_{i, \eps_2}| = \int_{B_{\eps_1}( 0)} | \zeta( x + \eps_1 z_i) - \zeta( x+ \eps_2 z_i ) | \, dx + \int_{B_{\eps_2}(0) \backslash B_{\eps_1}(0)} \zeta( x + \eps z_i) \, dx \\
\leq \| \nabla \zeta \|_{\infty} |\eps_2- \eps_1|  |z_i | \eps_1^{d} + \| \zeta\|_{\infty} \Bigl( \Bigl(\frac{\eps_2}{\eps_1}\Bigr)^d - 1 \Bigr) \eps_1^d.
\end{align*} 
Since $N^{\eps_2} (B)\leq N^{\eps_1}(B)$ and thus $i \leq N^{\eps_2}(B)$, we have that $| z_i | \leq \eps_2^{-1}$ and
\begin{align*}
|a_{i, \eps_1} - a_{i, \eps_2}|  \leq \|\zeta \|_{W^{1,\infty}} \biggl( \Bigl(1- \frac{\eps_1}{\eps_2}\Bigr) + \Bigl( \Bigl(\frac{\eps_2}{\eps_1}\Bigr)^d - 1 \Bigr) \biggr)\eps_1^d.
\end{align*}
Therefore, for the choice $\eps_1=\eps$, $\eps_2= \overline \eps$ this yields
\begin{align*}
|a_{i, \eps} - a_{i, \overline \eps}|  \leq \|\zeta \|_{W^{1,\infty}} \biggl ( \eps + \Bigl(\frac{1}{1-\eps}\Bigr)^d -1 \biggr) \overline\eps^d .
\end{align*}
and hence also the first limit in \eqref{l.w.independent} by \eqref{average.number} of Lemma \ref{SLLN.general.pp}. The second limit may be argued in a similar way. 
\end{proof}

 \section{Appendix}\label{App}
\subsection*{Strong Law of Large Numbers for sums of random variables with correlations. \,} This result is an easy adaptation to our setting, and to our needs, of the standard argument for the Strong Law of Large Numbers.

 \begin{lem}\label{SLLN.correlated}
Let $\{x_i\}_{i \in \N} = \Zd$, and let $\{ X_{i} \}_{i \in \N}$ be identically distributed random variables with $X_i \geq 0$ and $\langle X \rangle <+\infty$. 
Let us assume that for every $i,j \in \N$ with $i \neq j$
\begin{align}\label{SLLN.mixing.standard}
|\langle X_i X_j \rangle - \langle X \rangle^2 | < \frac{C}{|x_i-x_j|^\gamma} \ \ \ \ \ \gamma > d.
\end{align}
Then for every bounded Borel set $B \subset \Rd$ which is star-shaped with respect to the origin, we have
\begin{align}\label{SLLN.convergence}
\lim_{\eps \downarrow 0^+} \eps^{d}\sum_{x_i \in \Zd \cap {\frac 1 \eps}B} X_i = \langle X \rangle |B| \ \ \ \ \text{almost surely.}
\end{align}
\end{lem}
\begin{proof} 
The proof of this lemma is an easy adaptation of the standard argument for independent and identically distributed random variables: In particular, we adapt to the case of correlated variables the argument
of \cite[Subsection 2.4]{Durrett.probability.book}. 

Without loss of generality, we assume that $B= Q$, where $Q$ is the unitary cube centred at the origin and $Q^{\frac 1 \eps}= \frac 1 \eps Q$. 
Moreover, we may assume that $|x_i|$ is monotone in $i \in \N$. Thus, there exists a constant $c = c(d)$ such that for all $i \in \N$
\begin{align}
	\label{x_i.monotone}
	|x_i| \geq c i^{\frac 1 d}. 
\end{align}

\medskip

The first step is to reduce the study of \eqref{SLLN.convergence} to the sum of the truncated random variables $Y_i := X_i \1_{X_i \leq i}$.
Indeed, 
\begin{align}
	\sum_{i=1}^\infty \P(X_i > i) \leq \int_0^\infty \P(X > t) \dd t = \langle X \rangle < \infty.
\end{align}
Thus, by Borel-Cantelli theorem applied to the events $E_i:= \{ X_i > i \}$ we have that almost surely in \eqref{SLLN.convergence} we may substitute the variables $X_i$ with their truncated versions $Y_i$. Clearly, also the sequence $\{ Y_i\}_{i\in \N}$ satisfies \eqref{SLLN.mixing.standard}.

\medskip

We define $\tilde Y_i := Y_i - \langle Y_i \rangle$. The next step of the proof is the following estimate:
\begin{align}
	\label{sum.variances}
	\sum_{i=1}^\infty \frac{\langle \tilde Y_i^2 \rangle}{i^2} \leq 4 \langle X \rangle < \infty.
\end{align}
We estimate
\begin{align}
	\langle \tilde Y_i^2 \rangle \leq  \langle Y_i^2 \rangle  \int_0^\infty 2 y \P ( Y_i > y) \dd y \leq \int_0^i 2 y \P ( X > y) \dd y.
\end{align}
Using the monotone convergence theorem, this yields
\begin{align}
	\sum_{i=1}^\infty \frac{\langle \tilde Y_i^2 \rangle}{i^2} \leq \sum_{i=1}^\infty \frac{1}{i^2} \int_0^\infty \1_{(0,i)}(y) 2 y \P ( X > y) \dd y
	\leq  \int_0^\infty \sum_{i > y}^\infty \frac{1}{i^2} 2 y \P ( X > y) \dd y.
\end{align}
Since $\int_0^\infty \P ( X > y) \dd y = \langle X \rangle$, to prove \eqref{sum.variances} it suffices to show
\begin{align}
	y \sum_{i > y} \frac{1}{i^2} \leq 2.
\end{align}
If $y \geq 1$, then
\begin{align}
	y \sum_{i > y} \frac{1}{i^2}   = y \sum_{i= \lfloor y \rfloor+ 1}^\infty \frac{1}{i^2} \leq y \int_{ \lfloor y \rfloor}^\infty \frac{1}{t^2} \dd t = \frac{y}{ \lfloor y \rfloor} \leq 2.
\end{align}
If $ 0 < y < 1$,
\begin{align}
	y \sum_{i > y} \frac{1}{i^2}   \leq 1 + \sum_{i=2}^\infty \frac{1}{i^2} \leq 1 +  \int_{1}^\infty \frac{1}{t^2} \dd t = 2.
\end{align}
This concludes the proof of \eqref{sum.variances}.

\medskip

Next, we define 
$$
S_\eps:= \sum_{i \in \Zd \cap Q^{\frac 1 \eps}} Y_i, \ \ \ \ \tilde S_\eps:= \sum_{i \in \Zd \cap Q^{\frac 1 \eps}} \tilde Y_i.
$$
Then, for every $\delta >0$, we estimate by Chebyshev's inequality
\begin{align*}
\P(\eps^d \tilde S_\eps > \delta ) \leq  \eps^{2d}\frac{\langle \tilde S_\eps^{2} \rangle}{\delta^2} = \delta^{-2} \eps^{2d} \langle \sum_{j,i \in \Zd \cap Q^{\frac 1 \eps}} \tilde Y_i \tilde Y_j \rangle .
\end{align*}
By definition of $Y_i$ and assumption \eqref{SLLN.mixing.standard}  the last term is bounded by
\begin{align*}
\P(\eps^d \tilde S_\eps > \delta ) \leq \delta^{-2} \eps^{2d}\sum_{i \in \Zd \cap Q^{\frac 1 \eps}} \langle \tilde Y_i^2 \rangle + \delta^{-2} \eps^{2d} \sum_{j, i \in \Zd \cap Q^{\frac 1 \eps} \atop i \neq j}  \frac{C}{ |i - j|^\gamma}.
\end{align*}
We now restrict ourselves to consider the sequence $\eps_k:=\alpha^k$, $k\in \N$ and $\alpha \in (0, 1)$ and use the previous inequality to estimate
\begin{align}
	\label{sum.chebyshev}
\sum_{k=1}^{+\infty} \P(\eps_k^d \tilde S_{\eps_k} > \delta ) \leq \delta^{-2} \sum_{k=1}^{+\infty} \eps_k^{2d} \sum_{i \in \Zd \cap Q^{\frac{ 1 }{\eps_k}}} \langle \tilde Y_i^2 \rangle + \delta^{-2}\sum_{k=1}^{+\infty} \eps_k^{2d} \sum_{j, i \in \Zd \cap Q^{\frac{ 1}{ \eps_k}} \atop i \neq j}  \frac{C}{ |z_i - z_j|^\gamma}.
\end{align}
For the second term on the right-hand side above, thanks to assumption $\gamma > d$, we have
\begin{align}
	\label{sum.chebyshev.second}
\sum_{k=1}^{+\infty} \eps_k^{2d} \sum_{j\neq i \in \Zd \cap Q^{\frac 1 {\eps_k}} }  \frac{C}{ |z_i - z_j|^\gamma} \leq \sum_{k=1}^{+\infty} \eps_k^d < +\infty.
\end{align}
To estimate the first term on the right-hand side in \eqref{sum.chebyshev}, we can interchange the order of the sums since all terms are nonnegative.
Thus,
\begin{align*}
\sum_{k=1}^{+\infty} \eps_k^{2d} \sum_{i \in \Zd \cap Q^{\frac 1 {\eps_k}}} \langle \tilde Y_i^2 \rangle 
&= \sum_{i \in \Zd } \langle \tilde Y_i^2 \rangle \sum_{k=1}^{+\infty} \eps_k^{2d} \1_{x_i \in Q^{\frac 1 {\eps_k}}}
\stackrel{\eqref{x_i.monotone}}{\leq} \sum_{i \in \Zd } \langle \tilde Y_i^2 \rangle \sum_{k \colon \eps_k^{-d} \leq C i} \eps_k^{2d} \\
& \lesssim \sum_{i \in \Zd } \langle \tilde Y_i^2 \rangle \frac{1}{i^{2d}} \frac{1}{1-\alpha^{2d}} 
\stackrel{\eqref{sum.variances}}{\lesssim} \langle X \rangle < \infty
\end{align*}
Therefore, for every $\delta >0$  we have that $\sum_{k=1}^{+\infty} \P(\eps_k^d \tilde S_{\eps_k} > \delta ) < +\infty$ and by Borel-Cantelli's lemma and the Dominated Convergence theorem we get
\begin{align*}
\tilde S_{\eps_k} \to 0 \ \ \ \ \text{almost surely.}
\end{align*}
Since $\lim_{i \to \infty} \langle Y_i \rangle = \langle X \rangle$, this implies also
\begin{align}
\label{SLLN.subsequence}
S_{\eps_k} \to \langle X \rangle \ \ \ \ \text{almost surely.}
\end{align}

\medskip

To pass to the continuum limit $\eps \downarrow 0^+$ for the same full-probability set, we argue as in \cite{Durrett.probability.book} by monotonicity.
Indeed, for $\eps_{k+1} \leq \eps \leq \eps_k$, we have 
\begin{align*}
	 \Zd \cap Q^{\frac 1 {\eps_k}}
	\subset  \Zd \cap Q^{\frac 1 \eps}
\subset  \Zd \cap Q^{\frac {1} {\eps_{k+1}}}
\end{align*}
Hence, since $Y_i \geq 0$, it holds
\begin{align}
	\label{monotonicity}
	  \frac{\#\Bigl(Q^{\frac {1} {\eps_{k+1}}}\Bigr)}{\#\Bigl(Q^{\frac 1 {\eps_{k}}}\Bigr)} S_{\eps_{k}} \leq S_{\eps} \leq 
	\frac{\#\Bigl(Q^{\frac {1} {\eps_{k+1}}}\Bigr)}{\#\Bigl(Q^{\frac 1 {\eps_{k}}}\Bigr)} S_{\eps_{k+1}} .
\end{align}
By the choice $\eps_k=\alpha^k$, we obtain 
\begin{align*}
  \frac{\#\Bigl(Q^{\frac {1} {\eps_{k+1}}}\Bigr)}{\#\Bigl(Q^{\frac 1 {\eps_{k}}}\Bigr)}
 \to \alpha^{-d}.
\end{align*}
We now combine \eqref{SLLN.subsequence} and \eqref{monotonicity} to infer
\begin{align*}
 \liminf_{\eps \to 0} S_\eps \geq \alpha^d \langle X \rangle, 
 \qquad \limsup_{\eps \to 0} S_\eps \leq \alpha^{-d} \langle X \rangle.
\end{align*}
Since $\alpha \in (0,1)$ is arbitrary, we conclude the proof by sending $\alpha \to 1^-$ in the inequalities above.
\end{proof}

\subsection*{Conditions of Theorem \ref{t.main} for the processes defined in case (c) of Section \ref{setting}.}
By construction, the processes are stationary. Moreover, the marginal $\P_{\mathcal{R}}$ satisfies   \eqref{correlation.radii.0}, \eqref{correlation.radii.2}. Therefore, it suffices to prove that the point process $\Phi$ as defined in either (c.1) or (c.2) satisfies \eqref{finite.variance} and \eqref{mixing}.

\medskip

We begin with case (c.1): For a bounded set $D \subset \Rd$, $r >0$ and a point $x_i\in \Rd$, we define $N_1(D):= \#\Phi_1(D)$ and $N_{r}^i(D)= \#\Phi_r^{x_i} \cap (B_{r_i}(x_i) \cap D)$. For $R > 0$, let
$$
B_R(D):= \{ x \in \Rd \, \colon \, \mbox{dist}( x, D) \leq R\}.
$$
Then, by \eqref{matern}, we  estimate 
\begin{align*}
\langle N(Q)^2 \rangle &=   \biggl\langle \biggl( \sum_{z_i \in \Phi_1( \Rd ) } N_{r_i}^i( B_{r_i}(z_i) \cap Q) \biggr)^2  \biggr\rangle 
\leq  \biggl \langle N_1\bigl(  B_{R_c}(Q )\bigr) \sum_{z_i \in \Phi_1(  B_{R_c}(Q) ) } \bigl( N_{r_i}^i( B_{r_i}(z_i)) \bigr)^2 \biggr\rangle \\
&=  \langle N_1\bigr( B_{R_c} (Q)\bigr)^2 \rangle \langle N^0_{R_c}(B_{R_c}(0))^2 \rangle\leq  \lambda_1^2 \| \lambda_2\|^2_{\infty} R_c^{2d} |B_{R_c}(Q) |^2.
\end{align*} 
After taking the square-root of the above inequality, we conclude \eqref{finite.variance}.

\medskip

Condition \eqref{mixing} is an easy consequence of the fact that the process under consideration has finite range of dependence $R_c$, namely that if dist$(A, B) > R_c$,
then the random variables $N(A)$ and $N(B)$ are independent. Thus, condition \eqref{mixing} is satisfied for any $\gamma> 0$ and with a constant depending on $R_c$. 

\medskip

We now turn to case (c.2). In this case, property \eqref{finite.variance} is an immediate consequence of the choice $\beta < 1$ and the fact that, if $\tilde \Phi$ is the Poisson point process on $\Rd$ with intensity $\alpha$, then for every $m\in \N$ and bounded set $B\subset \Rd$
$$
\langle N(B)^m \rangle \leq \langle ( \#\tilde\Phi(B))^m \rangle .
$$
Furthermore, as in the previous case, the process considered in (c.2) has finite range of dependence given by $r_c$ and thus satisfies \eqref{mixing} for any $\gamma > 0$.

\section*{Acknowledgements}
The authors acknowledge support through the CRC 1060 \textit{(The Mathematics of Emergent Effects)} that is funded through the German Science Foundation (DFG), and the Hausdorff Center for Mathematics (HCM) at the University of Bonn.

\bibliographystyle{amsplain}
\bibliography{giunti.RS}
\end{document}